\documentclass[12pt]{article}
\usepackage{amssymb,amsfonts,amsmath,amsthm}
\usepackage[margin=0.71in]{geometry}
\usepackage{graphicx} 

\usepackage{xcolor}
\usepackage{caption}
\usepackage{subcaption}

\newtheorem{remark}{Remark}[section]
\newtheorem{theorem}{Theorem}[section]
\newtheorem{lemma}{Lemma}[section]
\newtheorem{proposition}{Proposition}[section]

\newtheorem{definition}{Definition}[section]
\newtheorem{assumption}{Assumption}[section]

\newcommand{\E}{\mathbb{E}}

\newcommand{\dd}{\,\mathrm{d}}

\newcommand{\commentout}[1]{}
\usepackage{soul} 

\title{A nonnegativity-preserving finite element method for a class of parabolic SPDEs with multiplicative noise}

\author{Ana Djurdjevac\thanks{University of Oxford, Mathematical Institute, Woodstock Road, Oxford OX2 6GG, UK, {\tt ana.djurdjevac@maths.ox.ac.uk}, Freie Universität Berlin, Arnimallee 6, 14195 Berlin, Germany, {\tt adjurdjevac@zedat.fu-berlin.de}},~ Claude Le Bris\thanks{{\'E}cole des Ponts and INRIA, 6-8 avenue Blaise Pascal, 77455 Marne La Vall{\'e}e, France, {\tt claude.le-bris@enpc.fr}},~ Endre S\"uli\thanks{University of Oxford, Mathematical Institute, Woodstock Road, Oxford OX2 6GG, UK, {\tt suli@maths.ox.ac.uk}}}

\date{}

\begin{document}

\maketitle

\begin{abstract}
 We consider a prototypical parabolic SPDE with finite-dimensional multiplicative noise, which, subject to a nonnegative initial datum, has a unique nonnegative solution. Inspired by well-established techniques in the deterministic case, we introduce a finite element discretization of this SPDE that is convergent and which, subject to a nonnegative initial datum and unconditionally with respect to the spatial discretization parameter, preserves nonnegativity of the numerical solution throughout the course of evolution. We perform a mathematical analysis of this method. In addition, in the associated linear setting, we develop a fully discrete scheme that also preserves nonnegativity, and we present numerical experiments that illustrate the advantages of the proposed method over alternative finite element and finite difference methods that were previously considered in the literature, which do not necessarily guarantee nonnegativity of the numerical solution.
\end{abstract}

\section{Introduction}\label{sec1}

We consider the following prototypical stochastic partial differential equation (henceforth abbreviated as SPDE) with multiplicative noise:
  \begin{equation}\label{eq:intro}
  \dd u - \Delta u \dd t = f(u) \dd W
  \end{equation}
(with the usual notation; see \eqref{eq:continous} below for the precise mathematical meaning) that preserves nonnegativity at the continuous level: the nonnegativity of an initial datum $u_0$ for \eqref{eq:intro} is propagated into the associated solution $u(t)$ for all $t\geq 0$.

The purpose of this article is to construct and analyze a finite element method (FEM) for this SPDE that is accurate and that preserves nonnegativity unconditionally in the discretization parameters (the time step $\Delta t$ and the spatial mesh size $h$).

Our motivation for this study stems from our interest in the Dean--Kawasaki equation and related equations, and more generally a large class of SPDEs that have nonnegative solutions because their solution models the \emph{density} of a system of possibly infinitely many particles (see \cite{djurdjevac2024weak, cornalba2023dean, martinez2024finite, magaletti2022positivity} for further details on this class of problems and for some illustrative and recent mathematical contributions to the topic). The present paper may be seen as a very preliminary step towards the design of suitable numerical methods for such equations. In any event, the equation we study and approximate numerically is sufficiently ubiquitous to be relevant in many practical settings. We therefore hope that the considerations herein will serve as a guideline for the construction of numerical methods for more sophisticated equations. 
The preservation of nonnegativity for parabolic SPDEs with multiplicative noise at the continuous level has also been explored in other contexts in several publications; see, for example, 
\cite{benth1997positivity,chen2019comparison,cresson2013positivity,Moreno}. Numerical methods for SPDEs have been extensively studied in the literature. For general results, we refer to \cite{MR3154916,lord2014introduction,walsh2005finite} and references therein. 
 
As was mentioned above, the specific feature of the SPDE we shall be focussed on here is nonnegativity-preservation. Our concern is therefore at the intersections of two different classes of questions, namely:
\begin{itemize}
\item[(i)] nonnegativity-preservation for finite element discretizations of deterministic partial differential equations; and
\item[(ii)] nonnegativity-preservation for time-stepping schemes for stochastic differential equations.
\end{itemize}
Even in the absence of noise, that is, for deterministic PDEs, preservation of nonnegativity, as stated in item (i) above, is a challenging requirement since it is intimately related to the construction of numerical methods that satisfy a discrete maximum principle;  see, for example, the recent monograph~\cite{MR4900691}. Since we cannot expect any simplifications in this respect in the presence of noise in an otherwise deterministic equation, our numerical method must comply with nonnegativity-preservation for PDEs. We will make these requirements precise later on (in Section~\ref{Section:MLFEM}). They concern the geometry of the mesh and the use of a specific version of the mass matrix, the so-called lumped mass matrix, in the finite element discretization, first introduced in the parabolic setting in \cite{fujii1973some} (see also~\cite{thomee2015positivity,chatzipantelidis2015preservation}).

Briefly, our choice of the spatial discretization is dictated by our wish to preserve nonnegativity. However, this comes at a price. A certain number of mathematical difficulties follow from the technical choices we are bound to make at the finite element level. Put differently, our analysis would have been much simpler and more standard to conduct had we opted for a discretization that is not required to systematically guarantee nonnegativity-preservation. To start with, in the absence of mass-lumping, we could have considered a formulation of the SPDE based on semi-group theory (a formulation that is routinely considered for generic SPDEs; see, for example, \cite{gyongy2009rate,gyongy2005discretization,tambue2019strong,kruse2014optimal,hausenblas2002numerical,yan2003finite}), and then our analysis of the discretization error would also have become much simpler and would have followed through the use of  standard techniques.
 
In addition, ignoring for the moment the dependence of the solution on the spatial variable and concentrating on the random variable, nonnegativity-preservation in the context of time-stepping schemes for SDEs is a well-established topic that has been of particular relevance in, for example, financial mathematics.
We will benefit from the knowledge accumulated in the field concerning this topic to formulate our time discretization; see \cite{scalone2022positivity,szpruch2011numerical,mao2021positivity,lei2023strong} for some specific contributions in this direction. In particular, we have been greatly inspired by the publications of A.~Alfonsi initiated in~\cite{alfonsi2005} and pursued, for example, in~\cite{alfonsi2010,alfonsi-book}.
 
In the context of SPDEs, publications that rigorously address nonnegativity-preservation at the fully-discrete level are rare. We are aware of only a handful of papers on this topic. They all address finite difference type methods (as opposed to finite element methods) and we shall only mention a couple: the recent work \cite{yang2022stochastic} introduces a nonnegativity-preserving finite difference scheme for the linear stochastic heat equation \emph{with finite-dimensional noise} (as will be the case in the present study). In the same spirit, the work~\cite{brehier2023analysis} develops a nonnegativity-preserving Lie–Trotter splitting scheme using a finite difference method in space for a nonlinear heat equation driven by multiplicative space-time white noise, and therefore improves on the former work in terms of generality of the noise but is restricted to the one-dimensional setting.  Positivity preservation for SPDEs has also been investigated in the recent work \cite{cui2023energy}, where energy-regularized logarithmic SPDE models and structure-preserving numerical approximations are studied.

It is important to note that the issue of nonnegativity-preservation cannot be overlooked. In the particular case of the Dean--Kawasaki equation, which we have mentioned as our motivation for the present work, it was shown in \cite{djurdjevac2024hybrid} that a classical spatial finite-volume approximation subject to a nonnegative initial condition may produce numerical solutions with negative values at subsequent times. Ad hoc techniques to overcome this shortcoming have therefore been developed  and then adapted to the  finite element  context in~\cite{martinez2024finite}.  This suffices to indicate the need for the development of a more systematic approach, which is our objective here.
 
It is also important to emphasize outright that our wish to consider finite element methods (as opposed to finite difference schemes), in addition to the fact that the former are of variational nature (and therefore more amenable to classical techniques of numerical analysis) and more flexible for domains with complicated geometry,  is related to \emph{legacy issues} in software engineering. Since many SPDEs of interest are parabolic and since  finite element methods are particularly efficient for the numerical solution of parabolic equations, it is interesting to explore how in-house-developed or off-the-shelf software for the numerical solution of PDEs can be adjusted, in an economical manner, for the numerical solution of parabolic SPDEs.

Our main contribution here is the construction of a semidiscrete finite element method for~\eqref{eq:intro}, based on continuous piecewise affine finite element basis functions in space, involving a lumped mass matrix. We prove that the method is first-order accurate and that it preserves nonnegativity, unconditionally with respect to the spatial mesh size. 
In addition, as a proof of concept for the general case and as an illustration of the possibility to extend our considerations from the semi-discrete level to a fully discrete scheme, we design, in the linear setting, a fully discrete scheme, combining the above  finite element method with a suitable time-stepping technique that also preserves nonnegativity.  

The article is organized as follows. Section~\ref{Sect:continousProblem} introduces in detail the initial-boundary-value problem for the equation \eqref{eq:intro} that we study; we recall the well-posedness of the problem and summarize qualitative properties of the analytical solution. The bulk of our article then consists of Sections~\ref{Section:MLFEM} through~\ref{Section:errorEstimates}, where we consider a spatially semi-discrete  finite element method. The method is introduced in Section~\ref{Section:MLFEM}; we prove there its well-posedness. The specific issue of nonnegativity-preservation is then the focus of Section~\ref{Section:nonnegativity}, while Section~\ref{Section:errorEstimates} contains the numerical analysis of the method. Our final section,  Section~\ref{Section:NumericalExperiments}, presents numerical experiments, focusing on the simple case of a linear parabolic SPDE. We construct a fully discrete scheme there by combining our semidiscrete FE discretization with a suitable time-stepping technique, corresponding to the issues (i) and (ii) highlighted above. We explore the performance of this fully discrete scheme and compare it, specifically in terms of nonnegativity-preservation, with some other available techniques.
 
\section{The continuous problem}\label{Sect:continousProblem}

Our objective is to discretize the following initial-boundary-value problem: 
\begin{alignat}{2}
\begin{aligned}\label{eq:contSPDE}
    \dd u - \Delta u \dd t 
     &= f(u) \dd W &&\qquad \text{in } D \times (0,T], \\
    u(\cdot,0) &= u_0(\cdot) &&\qquad \text{on } D, \\
    u |_{\partial D \times (0,T]} &= 0, \\
\end{aligned}
\end{alignat}
where $T>0$, $D \subset \mathbb{R}^d$, $d \in \{1,2,3\}$, is a bounded open domain and $\dd W$ denotes the noise.

Since our main goal here is the numerical approximation of this problem, for the sake of simplicity, we shall assume throughout the article that: 
\begin{itemize}
\item[(a)] the domain~$D$ is one of the following: a bounded open interval ($d=1$), a bounded open polygonal domain ($d=2$), a bounded open Lipschitz polyhedral domain ($d=3$), or a bounded open convex domain with a $C^2$ boundary ($d \in \{2,3\}$);  and that 
\item[(b)]
the noise $ W$ on the right-hand side of \eqref{eq:contSPDE}$_1$ is a \emph{truncated} $Q$-Wiener process, that is, a finite dimensional white-in-time noise $W^M = (W^M_t)_{t\geq 0}$ of the form 
\begin{equation}\label{noise_expansion}
    W^M(x,t) = \sum_{k=1}^M e_k(x) B_k(t),
\end{equation}
with a family $\{B_{k}\}_{k=1}^M$ of independent scalar Brownian motions defined on the probability space $(\Omega, \mathcal{F}, \{\mathcal{F}_t\}_{t \in [0,T]}, \mathbb{P})$, 
and $\{{e}_{k}\}_{k\in \mathbb{N}}$ form  an orthonormal basis of $L^2(D)$.  
\end{itemize}

This truncated noise is also used in~\cite{gyongy2005discretization,gyongy2009rate}, for example. 
For technical reasons we shall additionally assume throughout the article that $e_k \in C(\overline{D})$ for $k=1,\ldots,M$, and we define
\begin{equation}\label{eq:noise}
    c_e:= \max_{1 \leq k \leq M} \|e_k\|_{C(\overline{D})}.
\end{equation}
Extensions to general $Q$-Wiener processes could be considered, but we will not explore these here. Likewise, we consider on the left-hand side of \eqref{eq:contSPDE} the classical linear heat operator, but the ideas developed here apply to more general parabolic operators. A straightforward extension would be the addition of a zeroth-order term $c\,u$ to the left-hand side of the equation with a real-valued function $c \in L^\infty(0,T; C(\overline{D};\mathbb{R}_{\geq 0}))$. The inclusion of such a term does not affect our results and would only marginally impact our manipulations and proofs. For the sake of simplicity of the exposition we have therefore decided to omit it.  More general linear divergence-form uniformly elliptic operators could also be considered instead of $-\Delta$, the simplest example being $-\hbox{\rm div} [a(x)\,\nabla u]$, where $0 < c_0 \leq a(x) \leq c_1 < \infty$ for all $x \in D$, and $a$ is sufficiently smooth; again, for the sake of simplicity, we shall not discuss this situation here.

Regarding the right-hand side
of \eqref{eq:contSPDE}$_1$, we shall adopt the following assumption:
%
\begin{itemize}
\item[(c)] $f : \mathbb{R} \to \mathbb{R}$ is a uniformly Lipschitz continuous function such that $f(0)=0$; consequently, there exists a positive constant $c_f$ such that
\begin{equation}\label{eq:f}
    |f(s)| \leq c_f |s|\qquad \forall\, s \in \mathbb{R}.
\end{equation}
\end{itemize}

At least intuitively, the property $f(0) = 0$ is necessary in order to ensure the nonnegativity of the solution, to suppress the noise when the value of the solution is close to zero.

Our notion of solution to~\eqref{eq:contSPDE} stated in the next definition is one that is suitable for the construction of a finite element approximation of the problem.

\begin{definition}
A stochastic process with paths in $C([0,T];L^2(D)) \cap L^2(0,T; W^{1,2}_0(D))$ is a solution to the initial-boundary-value problem \eqref{eq:contSPDE}
for an initial condition~$u_0 \in L^2(D; \mathbb{R}_{\geq 0})$ if, almost surely,  
\begin{equation}\label{eq:continous}
    \int_D u(t) \varphi \dd x + \int_0^t \int_{D} \nabla u \cdot \nabla \varphi \dd x \dd s 
    = \int_D u_0 \varphi \dd x+ 
    \int_0^t \int_{D} f(u)  \varphi \dd x \dd W
\end{equation}
for every $\varphi \in W^{1,2}_0(D)$ and all $t \in [0,T]$.
\end{definition}

From \cite[Theorem 5.1.3]{liu2015stochastic} (see also \cite[Theorem 2.2]{gyongy2009rate} or \cite[Theorem 6.7]{da2014stochastic}) we infer the well-posedness of the problem under consideration. In particular, we have the following result.
 \begin{theorem}
     The problem \eqref{eq:contSPDE} has a unique solution. In addition, 
     \begin{equation}\label{energy_estimate_u}
         \mathbb{E} \left( \sup_{t \in [0,T]} \| u(t) \|^2_{L^2(D)}\right) < \infty.
     \end{equation}
 \end{theorem}

Having established well-posedness of the problem, we now turn to the question of nonnegativity preservation.
Under the additional assumption  that $f(0)=0$ (cf. \eqref{eq:f}), nonnegativity-preservation for solutions to the problem subject to a nonnegative initial datum $u_0 \in L^2(D;\mathbb{R}_{\geq 0})$ is, in fact, a consequence of the following more general \textit{comparison principle}, the proof of which is based on ingredients that are analogous to those of the comparison principle stated in \cite[Theorem 2.1]{donati1993white} and \cite[Theorem 2.5]{Kotelenez-1992}.
\begin{theorem}\label{thm_cont_comp}
    Let $\overline{u}$ and $\underline{u}$ be two solutions to \eqref{eq:contSPDE} with  corresponding initial data $\overline{u}_0, \underline{u}_0\in L^2(D)$, such that $\overline{u}_0(x) \geq \underline{u}_0(x)$ for Lebesgue almost all $x \in D$; then,
\begin{equation}
        \mathbb{P} \left[\overline{u}(x,t) \geq \underline{u}(x,t) \text{ for a.e. } x \in D \text{ and every }t \in [0,T] \right] =1. 
    \end{equation}
\end{theorem}

\section{The mass-lumped  finite element method}\label{Section:MLFEM}

With a view to preserving nonnegativity, we shall adopt here a specific version of the  finite element method. For deterministic PDEs, and more particularly for parabolic PDEs, a reference work that addresses issues in this direction is~\cite{fujii1973some}.  It is essentially established in that work that for the  finite element method with piecewise affine basis functions: (a) there are geometric conditions on the mesh to ensure that a discrete maximum principle holds (unconditionally with respect to the mesh size) and (b) that using mass-lumping in the variational formulation (as opposed to a conventional variational formulation) allows for this discrete maximum principle to hold for a larger range of time steps. In particular, in the case of a fully implicit time integrator (which will be our choice in our numerical illustrations in Section~\ref{Section:NumericalExperiments}), the property holds unconditionally with respect to the time step. Strictly speaking, the discrete maximum principle is a stronger property than nonnegativity-preservation at the discrete level, but, at least for a linear parabolic equation whose steady state is the null function, the two notions are equivalent. For all practical purposes, nonnegativity preservation is  therefore an appropriate requirement in the deterministic setting.

In order to formulate the mass-lumped finite element approximation of the initial-boundary-value problem \eqref{eq:contSPDE}, suppose that $\mathcal{T}_h$ is a shape-regular simplicial triangulation of $D$ with maximal edge-length (mesh size) $h>0$. A typical (closed)  simplex contained in $\mathcal{T}_h$ will be denoted by $K$. It is supposed that no simplex $K$ has more than $d$ vertices on $\partial D$. Let $D_h$ denote the interior of the closed set $\bigcup_{K \in \mathcal{T}_h} K$. 
When $D$ is a polygonal domain ($d=2$) or a polyhedral domain ($d=3$), then $D_h=D$. If on the other hand $D$ is a convex domain with a $C^2$ boundary, then $D_h \subset D$; we shall assume in that case that all vertices of the triangulation contained in $\partial D_h$ belong to $\partial D$. Let $\{x_i\,:\, i=1,\ldots, N\}$ be the set of all vertices of the triangulation $\mathcal{T}_h$ contained in the open set $D_h \subset D$; we consider
the finite element space 
\begin{equation}
    V_h:= \mbox{span}\{\phi_1,  \ldots, \phi_N\},
\end{equation}
spanned by the standard continuous piecewise affine finite element basis functions $\phi_j \in  W^{1,\infty}_0(D)$ on the triangulation $\mathcal{T}_h$, with $\mathrm{supp}(\phi_j) \subset \overline{D_h}\subset \overline{D}$, $j=1, \ldots, N$; then $\phi_j(x_i) = \delta_{i,j}$ for all $i,j=1,\ldots,N$.

Let $C_0(\overline{D})$ denote the set of all uniformly continuous functions defined on $\overline{D}$ that vanish on the boundary $\partial D$ of $D$. Given a function $w \in C_0(\overline{D})$, its  global continuous piecewise affine interpolant $\pi_h w$ is defined by
\begin{equation}\label{def:interpolant}
(\pi_h w)(x):= \sum_{K \in \mathcal{T}_h} \chi_K(x) (\pi_h w|_K)(x) ,\quad x \in \overline{D_h},
\end{equation}
where $\chi_K$ is the characteristic function of the (closed)  simplex $K \in \mathcal{T}_h$, $\pi_h w|_K$ is the local continuous piecewise affine interpolant of $w \in C(\overline{K})$ on the closed { simplex $K$: 
\[ (\pi_h w|_{K})(x) := \sum_{j=0}^d w(x^K_j) \phi_j^K(x),\quad x \in K,\]
where $x^K_j$, $j=0,\ldots, d$, are the vertices of the simplex $K$ and $\phi_j^K$, for $j\in \{0,\ldots, d\}$, is the affine basis function defined on $K$ associated with $x_j^K$, i.e., $\phi_i^K(x_j^K) = \delta_{i,j}$, $i, j=0,\ldots,d$. As $\pi_h w(x)=0$ for all $x \in \partial{D_h}$, it follows that $\pi_h w(x)=0$ for all $x \in \overline D \setminus D_h$.

Note that we can write
\begin{equation}
    (\pi_h w)(x) = \sum_{j=1}^N w(x_j) \phi_j(x), \quad x \in D.
\end{equation}

Hereafter, and in line with~\cite{fujii1973some} (and other publications on the subject, such as~\cite{thomee2015positivity} and \cite{MR4900691}) we assume that the triangulation is not only shape-regular, but also \emph{weakly acute}; that is, 
\[\int_  K \nabla \phi^K_i \cdot \nabla \phi^K_j \dd x \leq 0\quad \left\{
\begin{array}{ll}
\mbox{for all $i,j \in 
\{1,\ldots,N\}$ such that $i \neq j$}\\ 
\mbox{and for each simplex $K$ in the triangulation $\mathcal{T}_h$.}
\end{array}
\right.
\]

The  finite element method under consideration is then defined using the following \emph{mass-lumped variational formulation}: find $u_h(\cdot,t) \in V_h$ for $t \in (0,T]$, such that
\begin{equation}\label{ML-FEM_formulation}
    \begin{aligned}
(\pi_h[u_h(t) v_h], 1) &+ \int_0^t (\nabla u_h(s) , \nabla v_h) \dd s \\ 
& = (\pi_h[u_h^0 v_h], 1) + \int_0^t (\pi_h [f(u_h(s)) v_h \dd W^M], 1) \dd s\quad \forall\,   v_h \in V_h,
\end{aligned}
\end{equation}
where, as above, $(\cdot,\cdot)$ is the inner product of $L^2(D)$, and the initial condition $u_h^0 \in V_h$ is defined by
\[u_h^0:= \pi_h u_0,\] 
with $u_0 \in C_0(\overline{D};\mathbb{R}_{\geq 0})$ given.

\subsection{An equivalent system of SDEs}

Because  the stochastic process $u_h(\omega,\cdot,t) \in V_h$ for all $\omega \in \Omega$ and $t \in [0,T]$, it can be represented as $u_h(\omega,x,t) = \sum_{j=1}^N U_j(\omega,t) \phi_j(x)$ for $\omega \in \Omega$ and $(x,t) \in \overline{D} \times [0,T]$. For the sake of notational simplicity, here and hereafter we shall suppress the argument $\omega \in \Omega$ and write $U_j(t)$ instead of $U_j(\omega,t)$. We define $U_j(0)=U_j^0:=u_0(x_j)$, $j=1,\ldots,N$. We note in particular that, by definition, $u_h(x,t)=0$ for all $(x,t) \in (\overline{D}\setminus D_h) \times [0,T]$. The purpose of this subsection is to restate the finite element method, defined in \eqref{ML-FEM_formulation} above, as a system of SDEs for the stochastic vector function $U$ with values in $\ C([0,T];\mathbb{R}^N)$, where $(U(t))_j:=U_j(t)$, $j=1,\ldots,N$.
Because $\phi_j(x_i) = \phi_i(x_j) = \delta_{i,j}$, $i,j = 1,\ldots,N$, we have that 
\[ (\pi_h[u_h(t) \phi_i], 1) 
= \int_D \sum_{j=1}^N  (u_h(x_j,t)\phi_i(x_j) )\phi_j(x) \dd x = u_h(x_i,t) \int_D \phi_i(x) \dd x\]
for $i=1,\ldots,N$, and therefore 
\begin{align} \label{eq:L1}
(\pi_h[u_h(t) \phi_i], 1) = U_i(t) \int_D \phi_i(x) \dd x, \quad i=1,\ldots,N.
\end{align}

Next, we consider, for $i=1,\ldots,N$,
\[ ( (\nabla u_h(s) \cdot \nabla \phi_i), 1)  
= \sum_{j=1}^N U_j(s) \int_D \nabla \phi_j(x) \cdot \nabla \phi_i(x) \dd x.\] 
We define the matrix $A \in \mathbb{R}^{N \times N}$ by 
\begin{equation}
\label{eq:matrixA}
A_{i,j}:= \frac{\int_D \nabla \phi_j(x) \cdot \nabla \phi_i(x) \dd x}{\int_D \phi_i(x) \dd x },\quad i,j=1,\ldots,N.
\end{equation}
Clearly $A=A^{\mathrm{T}}$ if and only if $\int_D \phi_i \dd x = \int_D \phi_j \dd x$ for all $i,j \in \{1,\ldots,N\}$. Thus, in general, the matrix $A$ is \textit{not} symmetric. In any case, because the numerator of $A_{i,j}$ is equal to $A_{i,j} \int_D \phi_i(x) \dd x$, we have
\[ ((\nabla u_h(s) \cdot \nabla \phi_i), 1) = \left(\int_D \phi_i(x) \dd x \right) \sum_{j=1}^N A_{i,j} U_j(s).\]
Recalling that, by definition, $U_j(s) \in \mathbb{R}$ is the $j$-th component of the vector $U(s) \in \mathbb{R}^N$, it follows that
\begin{align} \label{eq:L2}
((\nabla u_h(s) \cdot \nabla \phi_i), 1) = \left(\int_D \phi_i(x) \dd x \right) (AU(s))_i, \quad i=1,\ldots, N.
\end{align}

Finally, we have that 
\begin{align}\label{eq:R2}
(\pi_h [f(u_h(s)) \phi_i \dd W^M(s)], 1)
&= \sum_{k=1}^M \left(\int_D \pi_h [f(u_h(x,s)) \phi_i(x) e_k(x)] \dd x \right)\! \dd B_k(s)\nonumber\\
&=\sum_{k=1}^M \sum_{j=1}^N f(u_h(x_j,s))\phi_i(x_j) e_k(x_j)\left(\int_D \phi_j(x)\dd x \right)\! \dd B_k(s)\nonumber\\
&=\sum_{k=1}^M f(u_h(x_i,s)) e_k(x_i) \left(\int_D \phi_i(x)\dd x \right)\! \dd B_k(s)\nonumber\\
&=f(U_i(s)) \left(\int_D \phi_i(x)\dd x \right) \sum_{k=1}^M e_k(x_i) \dd B_k(s)\nonumber\\
&= \left(\int_D \phi_i(x)\dd x \right) \sum_{k=1}^M (f(U(s)) E_k)_i \dd B_k(s),\quad i=1,\ldots,N,
\end{align}
where, for each $k=1,\ldots,M$, $f(U(s))E_k \in \mathbb{R}^N$ is a vector whose $i$-th component is defined to be $f(U_i(s))e_k(x_i) = f(u_h(x_i,s))e_k(x_i)=\pi_h[f(u_h(x,s)e_k(x)]|_{x=x_i}$,  $i=1,\ldots,N$, $s \in (0,T]$.
Taking $v_h = \phi_i$, $i=1,\ldots,N$, in \eqref{ML-FEM_formulation} it follows from \eqref{eq:L1}, \eqref{eq:L2} and \eqref{eq:R2} that the  finite element method \eqref{ML-FEM_formulation} can be rewritten in the following equivalent form:
\begin{equation}\label{SDE_formulation}
    U(t) + \int_0^t AU(s) \dd s 
    = U^0 + \int_0^t\sum_{k=1}^M  (f(U(s)) E_k) \dd B_k(s),\quad t \in [0,T].
\end{equation}
We shall use in the sequel both \eqref{ML-FEM_formulation} and \eqref{SDE_formulation}, depending on which of these equivalent forms is more convenient to work with in the specific context.

\subsection{Well-posedness and energy estimate}

In order to apply It$\hat{\rm o}$'s formula, we introduce the seminorm $ \| \cdot \|_h$ on $C_0(\overline{D})$ defined by
\begin{equation}
\label{eq:discrete-norm}
    \| w \|_h := \| \pi_h(w^2) \|^{1/2}_{L^1(D)},\quad w \in C_0(\overline{D}),
\end{equation}
which is a norm on the finite element space $V_h$. The norm $\|\cdot\|_h$ is induced by the semi-inner product $(\cdot,\cdot)_h$ defined by
\begin{equation}
\label{eq:discrete-scalar}
    (w, v)_h := (\pi_h(w v), 1), \quad w \in C(\overline{D}),\, v \in C_0(\overline{D}),
\end{equation}
which is an inner product on $V_h$.

The function space in which we shall work is $(V_h, (\cdot, \cdot)_h)$. 
To proceed, we recall the definition of the discrete Dirichlet Laplacian $-\Delta_h : V_h \to V_h$:
\begin{equation}\label{def:Delta_h-h}
    (-\Delta_h w_h, v_h)_h := (\nabla w_h, \nabla v_h), \quad  w_h, v_h \in V_h,
\end{equation}
and note that 
\begin{equation}\label{L2-projection-h}
    (\pi_h v, w_h)_h = (\pi_h((\pi_h v )w_h), 1) = (\pi_h(v w_h),1)  = (v, w_h)_h, \quad v \in C_0(\overline{D}),\,  w_h \in V_h,
\end{equation}
where the first and the third equalities come from the definition of the inner product $(\cdot,\cdot)_h$, while the second equality is a consequence of the fact  that the $\pi_h((\pi_h v)w_h) = \pi_h(uw_h)$ because $\pi_h((\pi_hv)w_h)$ and $\pi_h(vw_h)$ are continuous piecewise affine functions on $\mathcal{T}_h$ and $((\pi_hv)w_h)(x) = (\pi_hv)(x) w_h(x) = v(x)w_h(x)$ at each vertex $x$ of the triangulation $\mathcal{T}_h$, -- and there is a unique such continuous piecewise affine function defined on $\mathcal{T}_h$.
In particular, \eqref{L2-projection-h} implies that $\pi_h: C_0(\overline{D}) \to V_h$ is an orthogonal projector onto $V_h$ with respect to the inner product $(\cdot,\cdot)_h$, and $\|\pi_h v\|_h \leq \|v\|_h$ for all $v \in C_0(\overline{D})$.
Furthermore, note that $\pi_h$ is idempotent, i.e., $\pi_h (\pi_h v) = \pi_h v$ for all $v \in C_0(\overline{D})$. 

Using the inner product $(\cdot,\cdot)_h$, we have that
\begin{align*}
    (\pi_h[f(u_h) e_k v_h],1) = (f(u_h)e_k, v_h)_h = (\pi_h(f(u_h)e_k), v_h)_h \quad \forall\, v_h \in V_h.
\end{align*}
We can therefore rewrite \eqref{ML-FEM_formulation} in the following equivalent form:
\begin{align}\begin{aligned}\label{FEM-formulation-h-product}
    &(u_h(t), v_h)_h + \int_0^t (-\Delta_h u_h(s), v_h)_h \dd s \\ &\quad = 
    (u_h^0, v_h)_h + \int_0^t \sum_{k=1}^M (\pi_h(f(u_h(s))e_k), v_h)_h \dd B_k(s),
\end{aligned}\end{align}
or, in a more concise form, as the following equation in $V_h$:
\begin{equation}\label{FEM-formulation-strong}
    \dd u_h - \Delta_h u_h \dd t = \sum_{k=1}^M \pi_h(f(u_h)e_k) \dd B_k(t), \quad \mbox{with $u_h(0)=u_h^0:=\pi_h u_0$.}
\end{equation}
We then have the following well-posedness result. 

\begin{proposition}\label{w-p-uh} Suppose that $u_0 \in C_0(\overline{D};\mathbb{R}_{\geq 0})$; then, 
  there exists a unique solution $u_h \in C([0,T]; L^2(\Omega, V_h))$ to~\eqref{FEM-formulation-strong}. Moreover, the following energy estimate holds:
\begin{equation}\label{eq:est11}
     \E [\| u_h(t) \|_h^2] + 2 \int_0^t \E [\| \nabla u_h (s)\|^2 \dd s] \leq \exp(t c_e^2 c_f^2 M ) \, \E  [\| u_h^0 \|_h^2] \quad \forall\, t \in (0,T].
\end{equation}
\end{proposition}
\begin{proof}
     The existence of a unique solution $u_h \in C([0,T]; L^2(\Omega, V_h))$ to \eqref{SDE_formulation} (which is equivalent to~\eqref{FEM-formulation-strong}) follows from standard theory for systems of SDEs with Lipschitz continuous coefficients.     
     The energy estimate~\eqref{eq:est11} is obtained using  It$\hat{\rm{o}}$'s formula from~\cite[Lemma 4.1]{MR651582} 
     over the (finite-dimensional) Hilbert space $(V_h, \| \cdot \|_h)$.    
     Hence, we have that
    \begin{align*}
        \| u_h(t) \|^2_h = \| u_h^0 \|^2_h &+ 2 \int_0^t (\Delta_h u_h(s), u_h(s))_h \dd s + \sum_{k=1}^M \int_0^t \| \pi_h(f(u_h(s))e_k) \|^2_h \dd s  \\
        &+ 2 \sum_{k=1}^M \int_0^t ( \pi_h(f(u_h(s))e_k), u_h(s))_h \dd B_k(s).
    \end{align*}
    As $\| \pi_h(f(u_h(s))e_k) \|_h \leq \| f(u_h(s))e_k \|_h$, it follows from our assumption \eqref{eq:f} on $f$, and the assumed boundedness \eqref{eq:noise} of the functions $e_k$, $k=1,\ldots,M$, and the  nonnegativity of the finite element basis functions, that
    \begin{equation*}
       \| \pi_h(f(u_h(s))e_k) \|_h^2 \leq c_e^2 c_f^2 \|u_h(s)\|_h^2.
    \end{equation*}
    It then follows from the definition of $-\Delta_h$ that
    \begin{align*}
        \| u_h(t) \|^2_h + 2\int_0^t \|\nabla u_h(s) \|^2  \dd s \leq \| u_h^0 \|^2_h 
        + c_e^2 c_f^2 \sum_{k=1}^M \int_0^t \| u_h(s) \|_h^2 \dd s + M_t, 
    \end{align*}
    where $M_t:=  2 \sum_{k=1}^M \int_0^t ( \pi_h(f(u_h(s))e_k), u_h(s))_h \dd B_k(s)$ 
    is a martingale. To confirm this, we note that 
    \[\int_0^t |( \pi_h(f(u_h(s))e_k), u_h(s))_h|^2 \dd s \leq c_e^2 c_f^2  \int_0^t \|u_h(s)\|_h^4\dd s < \infty\]
    for any fixed $h>0$, where we have made use of the fact that $u_h \in C([0,T]; V_h)$, almost surely. Furthermore, since $u_h$ solves a system of SDEs with Lipschitz continuous coefficients, we know that all moments are bounded, which implies that $M_t$ is a martingale.
    Hence, by taking the expectation of the last displayed inequality, we obtain the bound
    \begin{align}\label{energy-with-gradient}
         \E [\| u_h(t) \|_h^2 ]  + 2 \int_0^t \E[\|\nabla u_h(s) \|^2] \dd s\leq 
         \E[\| u_h^0 \|_h^2] + c_e^2 c_f^2 M \int_0^t \E[\| u_h(s) \|_h^2] \dd s.
    \end{align}
    We then obtain~\eqref{eq:est11} from \eqref{energy-with-gradient}, using the Gr\" onwall's lemma.
\end{proof}

\section{The nonnegativity-preservation property}\label{Section:nonnegativity}

\subsection{It$\hat{\bf o}$'s formula for mass-lumping}

The formulation \eqref{SDE_formulation} is well suited for applying It$\hat{\rm{o}}$'s formula \cite[Theorem 1.2]{pardouxt1980stochastic}, \cite[Theorem 2.1]{donati1993white}, which in turn will be crucial in our proof of preservation of nonnegativity. In fact, we have the following result.

\begin{lemma}
\label{lemma:new-ito}
    Suppose that    
     $\psi \in C^{2}(\mathbb{R})$, with $\psi(0)=\psi'(0)=0$,
     and that $u_h$ solves \eqref{FEM-formulation-h-product}, with $u_0 \in C_0(\overline{D};\mathbb{R}_{\geq 0})$; then, the following equality holds:
\begin{align}\label{eq:new-ito}
&\int_D \pi_h[\psi(u_h(x,t))] \dd x + \int_0^t \left[\int_D \nabla u_h(x,s) \cdot \nabla \left(\pi_h[\psi'(u_h(x,s))]\right) \dd x \right]\! \dd s\nonumber \\
&
= \int_D \pi_h[\psi(u_h^0(x))] \dd x\nonumber \\
&\quad +\sum_{k=1}^M \int_0^t \int_D \pi_h [\psi'(u_h(x,s))f(u_h(x,s))e_k(x)] \dd B_k(s)\nonumber \\
&\quad + 
\frac{1}{2}\sum_{k=1}^M \int_0^t \int_D \pi_h[\psi''(u_h(x,s))(f(u_h(x,s))e_k(x))^2]\dd x \dd s,\quad t \in [0,T].
\end{align}
\end{lemma}

\begin{proof}
    Let us define the functional  $V \in \mathbb{R}^N \mapsto \Psi(V) \in \mathbb{R}_{\geq 0}$ by
\[ \Psi(V):= \int_D \pi_h [\psi(v_h(x))]\dd x, \]
where $v_h \in V_h$ is defined by $v_h(x):= \sum_{j=1}^N V_j \phi_j(x)$. In other words, we associate to any~$V \in \mathbb{R}^N$ a function $v_h \in V_h$ through $v_h(x):= \sum_{j=1}^N V_j \phi_j(x)$, and then use this $v_h$ to define $\Psi(V)$ through $\int_D \pi_h[\psi(v_h)] \dd x$.
By It$\hat{\rm{o}}$'s formula we have, for $t \in [0,T]$,
\begin{align}
\begin{aligned}
\label{eq:ito-111}
 \Psi(U(t)) &+ \int_0^t \langle \Psi'(U(s)), AU(s) \rangle \dd s 
 = \Psi(U^0) + \sum_{k=1}^M \int_0^t \langle \Psi'(U(s)), f(U(s))E_k\rangle \dd B_k(s)\\ &+ \frac{1}{2}\sum_{k=1}^M \int_0^t \langle \Psi''(U(s))f(U(s))E_k, f(U(s))E_k \rangle \dd s, 
\end{aligned}
\end{align}
where $\Psi'$ and $\Psi''$ are, respectively, the first and second Gateaux derivative of $\Psi$.
We shall now rewrite~\eqref{eq:ito-111} in terms of the original function $u_h$. To this end, we successively consider all the terms of~\eqref{eq:ito-111}.

The first term on the left-hand side of~\eqref{eq:ito-111} reads 
\begin{equation}
\label{eq:first-left}
\Psi(U(t)) = \int_D \pi_h[\psi(u_h(x,t)] \dd x, \quad t \in [0,T].
\end{equation}

To deal with the second term, we need to compute $\Psi'(U)$. Thanks to the definition of the Gateaux derivative, 
\begin{align*}
\langle \Psi'(U),V\rangle &= \frac{\dd}{\dd \alpha} \Psi(U+\alpha V)|_{\alpha=0}
=\int_D \pi_h[\psi'(u_h(x))v_h(x)] \dd x,
\end{align*}
where $u_h(x):=\sum_{j=1}^N U_j \phi_j(x)$ and $v_h(x):=\sum_{j=1}^N V_j \phi_j(x)$ for $U, V \in \mathbb{R}^N$.
Consequently, 
\begin{align*}
\langle \Psi'(U(s)), AU(s) \rangle =  \int_D \pi_h[\psi'(u_h(x,s))v_h(x,s)] \dd x,
\end{align*}
where $v_h(x,s) = \sum_{i=1}^N (AU(s))_i\, \phi_i(x) = \sum_{i,j=1}^N A_{i,j}U_j(s)\, \phi_i(x)$. Next, 
\begin{align*} 
\pi_h[\psi'(u_h(x,s))v_h(x,s)]&= \sum_{\ell=1}^N \psi'(u_h(x_\ell,s))v_h(x_\ell,s)\phi_\ell(x) \\
&= \sum_{\ell=1}^N \psi'(U_\ell(s)) \left(\sum_{i,j=1}^N A_{i,j}U_j(s) \phi_i(x_\ell) \right)\phi_\ell(x)\\
&= \sum_{\ell=1}^N\psi'(U_\ell(s))\left(\sum_{j=1}^N A_{\ell,j}U_j(s) \right)\phi_\ell(x) \\
&= \sum_{i=1}^N\psi'(U_i(s))\left(\sum_{j=1}^N A_{i,j}U_j(s) \right)\phi_i(x).
\end{align*}
Therefore, recalling the definition of $A_{i,j}$, we have the following equality:
\begin{align*}
\int_D \pi_h[\psi'(u_h(x,s))v_h(x,s)]
 \dd x &= \sum_{i=1}^N \psi'(U_i(s))\left(\sum_{j=1}^N A_{i,j}U_j(s) \right) \int_D \phi_i(x) \dd x\\
 &= \sum_{i=1}^N\psi'(U_i(s))\sum_{j=1}^N \left(\int_D \nabla \phi_j(x) \cdot \nabla \phi_i(x) \dd x \right) U_j(s)\\
 &=\int_D \nabla\left(\sum_{j=1}^N U_j(s) \phi_j(x)\right) \cdot \nabla \left(\sum_{i=1}^N \psi'(U_i(s))\phi_i(x) \right)\! \dd x\\
 &= \int_D \nabla u_h(x,s) \cdot \nabla \left(\sum_{i=1}^N \psi'(u_h(x_i,s))\phi_i(x) \right) \!\dd x\\
 &= \int_D \nabla u_h(x,s) \cdot \nabla \left(\pi_h[\psi'(u_h(x,s))]\right) \!\dd x.
\end{align*}
We have thus shown that
\begin{equation}
\label{eq:second-left}
 \int_0^t \langle \Psi'(U(s)), AU(s)\rangle\dd s = \int_0^t \left[\int_D \nabla u_h(x,s) \cdot \nabla \left(\pi_h[\psi'(u_h(x,s))]\right) \!\dd x \right]\!\! \dd s,\quad t \in [0,T].
 \end{equation}

Next, we turn to the right-hand side of~\eqref{eq:ito-111}. To start with, we notice that
\begin{equation}
\label{eq:first-right}
 \Psi(U^0) = \int_D \pi_h[\psi(u_h^0(x))]\dd x.
 \end{equation}

Using a similar calculation as above, we also have that
\begin{equation}
\label{eq:second-right}
\sum_{k=1}^M \int_0^t \langle \Psi'(U(s)), f(U(s))E_k\rangle \dd B_k(s) = \sum_{k=1}^M \int_0^t \int_D \pi_h [\psi'(u_h(x,s))f(u_h(x,s)e_k(x)] \dd B_k(s)
 \end{equation}
for all $t \in [0,T]$.

It remains to transform the rightmost term of~\eqref{eq:ito-111}
\[ \frac{1}{2}\sum_{k=1}^M \int_0^t \langle \Psi''(U(s))f(U(s))E_k, f(U(s))E_k \rangle \dd s\]
in It$\hat{\mathrm{o}}$'s formula, written in terms of $U$, into an equal expression written in terms of $u_h$. To do this, we need to calculate $\Psi''$:
\[ \langle \Psi''(U)V, W \rangle = \frac{\dd}{\dd \alpha} \langle \Psi'(U + \alpha W) , V \rangle|_{\alpha = 0} = \int_D \pi_h[\psi''(u_h)v_h w_h]\dd x,\] 
where $u_h(x,t) = \sum_{j=1}^N U_j(t) \phi_j(x)$, $v_h(x) = \sum_{j=1}^N V_j \phi_j(x)$ and $w_h(x) = \sum_{j=1}^N W_j \phi_j(x)$, for $V, W \in \mathbb{R}^N$. Therefore,
and because 
\begin{align*} 
\sum_{j=1}^N (f(U(s))E_k)_j \phi_j(x) &= \sum_{j=1}^N f(U_j(s))e_k(x_j) \phi_j(x) \\
&= \sum_{j=1}^N f(u_h(x_j,s))e_k(x_j) \phi_j(x)= \pi_h[f(u_h(s))e_k](x),
\end{align*}
we have $\pi_h[f(u_h(\cdot,s))e_k](x_i) = (f(U(s))E_k)_i$, and hence
\begin{align}\label{eq:rightmost-right}
&\frac{1}{2}\sum_{k=1}^M \int_0^t \langle \Psi''(U(s))f(U{(s)})E_k, f(U)E_k \rangle \dd s \nonumber\\\nonumber
&= 
\frac{1}{2}\sum_{k=1}^M \int_0^t \int_D\! \pi_h[\psi''(u_h(x,s))(\pi_h[f(u_h(x,s))e_k(x)])^2] \dd x \dd s \\
&= 
\frac{1}{2}\sum_{k=1}^M \int_0^t \int_D \pi_h[\psi''(u_h(x,s))(f(u_h(x,s))e_k(x))^2]\dd x \dd s.
\end{align}
Collecting \eqref{eq:first-left} through~\eqref{eq:rightmost-right} and substituting into \eqref{eq:ito-111}, we obtain \eqref{eq:new-ito}; that concludes the proof of Lemma~\ref{lemma:new-ito}.
\end{proof}

\subsection{Nonnegativity-preservation}

We now proceed to prove (almost sure) nonnegativity-preservation for the numerical solution.
\begin{theorem}
\label{thm:nonnegativity}
   Suppose that $u_0 \in C_0(\overline{D};\mathbb{R}_{\geq 0})$ and let $u_h$ denote the corresponding solution to \eqref{FEM-formulation-strong}; then,
\begin{equation}
        \mathbb{P} \left[u_h(x,t) \geq 0\, \,\mbox{for all $(x,t) \in \overline{D} \times [0,T]$} \right] = 1.
    \end{equation}
\end{theorem}
\begin{proof}
     We first define a convex function $\varphi_p \in C^2(\mathbb{R};\mathbb{R}_{\geq 0})$ whose first derivative is a regularization of the function $s \in \mathbb{R} \mapsto [s]_{-}:= \min\{s,0\} \in \mathbb{R}_{\leq 0}$. More precisely, for $p > 0$, we consider
\[ \varphi_p(s) := \begin{cases} 
      \frac{s^2}{2} + \frac{s}{2p} + \frac{1}{6p^2}, & s\leq - \frac{1}{p}, \\
      -\frac{ps^3}{6}, & -\frac{1}{p}\leq s\leq 0, \\
      0, & s \geq 0. 
   \end{cases}
\]
The first and second derivatives of $\varphi_p$ are then given, respectively, by
\[
\varphi_p'(s) = \begin{cases} 
      s + \frac{1}{2p},  & s\leq - \frac{1}{p}, \\
      -\frac{ps^2}{2}, & -\frac{1}{p}\leq s\leq 0, \\
      0, & s \geq 0; 
   \end{cases}
   \qquad 
\varphi_p''(s) = \begin{cases} 
      1,  & s\leq - \frac{1}{p}, \\
      -ps, & -\frac{1}{p}\leq s\leq 0, \\
      0, & s \geq 0. 
   \end{cases}
\]

We now choose $\psi=\varphi_p$ as test function in~\eqref{eq:new-ito} and take the expectation of both sides of the resulting equality; this gives
\begin{align}
\begin{aligned}\label{newnewIto}
&\mathbb{E}\int_D \pi_h[\varphi_p(u_h(x,t))] \dd x +\mathbb{E} \int_0^t \left[\int_D \nabla u_h(x,s) \cdot \nabla \left(\pi_h[\varphi_p'(u_h(x,s))]\right) \dd x \right]\! \dd s\\
&\,= \mathbb{E}\int_D \pi_h[\varphi_p(\pi_h u_0(x))] \dd x + \frac{1}{2} \,\mathbb{E}
\sum_{k=1}^M \int_0^t \int_D \pi_h[\varphi_p''(u_h(x,s))(f(u_h(x,s))e_k(x))^2]\dd x \dd s.
\end{aligned}
\end{align}
Since, by assumption,  $u_0(x)\geq 0$ for all $x \in \overline{D}$, also $\pi_h u_0(x) \geq 0$ for all $x \in \overline{D}$; thus,
\[ \mathbb{E}\int_D \pi_h[\varphi_p(\pi_h u_0(x))] \dd x = 0.\]
For the second term on the left-hand side of \eqref{newnewIto}, since $\varphi_p'$ is a monotonically increasing globally Lipschitz continuous function, we can invoke 
\cite[Lemma 4.1]{barrett2012finite}
to infer that
\[\mathbb{E} \int_0^t \left[\int_D \nabla u_h(x,s) \cdot \nabla \left(\pi_h[\varphi_p'(u_h(x,s))]\right) \!\dd x \right]\! \dd s\,\geq \,0.\]
This then implies that 
\begin{align*}\begin{aligned}
&\mathbb{E}\int_D \pi_h[\varphi_p(u_h(x,t))] \dd x \leq \frac{1}{2} \,\mathbb{E}
\sum_{k=1}^M \int_0^t \int_D \pi_h[\varphi_p''(u_h(x,s))(f(u_h(x,s))e_k(x))^2]\dd x \dd s
\end{aligned}
\end{align*}
for all $t \in [0,T]$.
In view of \eqref{eq:f} and \eqref{eq:noise}, we obtain the following bound: 
\begin{align}
\begin{aligned}\label{Ito4}
\mathbb{E}\, \int_D \pi_h (\varphi_p(u_h(x,t)))\dd x\leq \frac{1}{2} M\,c_e^2\, c_f^2 \,\mathbb{E}
\int_0^t \int_D \pi_h[\varphi_p''(u_h(x,s))(u_h(x,s))^2]\dd x \dd s,\quad t \in [0,T].
\end{aligned}
\end{align}
Clearly, $\varphi_p''(s)s^2 = 0 = \varphi_p(s)$ for all $s \geq 0$ and all $p > 0$. Similarly, $\varphi_p''(s)s^2 = 6\varphi_p(s)$ for all $s \in [-1/p,0]$ and all $p > 0$. A simple calculation also shows that $\varphi_p''(s)s^2 \leq 6 \varphi_p(s)$ for all $s \leq -1/p$ and all $p > 0$; this follows by noting that $6 \varphi_p(s) - \varphi_p''(s)s^2 = 2(s+1/p)(s+1/(2p)) \geq 0$ for all $s\leq -1/p$ and all $p > 0$. Hence, $0 \leq \varphi_p''(s)s^2 \leq 6 \varphi_p(s)$ for all $s \in \mathbb{R}$ and all $p > 0$. Using this inequality on the right-hand side of \eqref{Ito4} we deduce (assuming henceforth that $p > 0$) that
\begin{align}
\begin{aligned}\label{Ito5}
\mathbb{E}\, \int_D \pi_h (\varphi_p(u_h(x,t)))\dd x\leq 3 M\,c_e^2\, c_f^2 \,\mathbb{E}
\int_0^t \int_D \pi_h[\varphi_p(u_h(x,s))]\dd x \dd s,\quad t \in [0,T].
\end{aligned}
\end{align}
Defining 
\[X_p(t):= \mathbb{E}\, \int_D \pi_h (\varphi_p(u_h(x,t)))\dd x,\] 
inequality \eqref{Ito5} can be rewritten as
\[X_p(t)\leq 3 M\,c_e^2\, c_f^2 \, \int_0^t X_p(s)\dd s.\]
Observing that by construction $X_p(t)\geq 0$, Gr\"onwall's lemma implies that $X_p(t)=0$ for all $t \in [0,T]$; thus,
\begin{align}\label{eq-zero}
\mathbb{E}\, \int_D \pi_h (\varphi_p(u_h(x,t)))\dd x = 0\quad \mbox{for all $t \in [0,T]$}.
\end{align}
We shall now pass to the limit $p \rightarrow +\infty$, using Lebesgue's dominated convergence theorem. To this end, we need to check that the hypotheses of Lebesgue's dominated convergence hold.  
We begin by noting that~$\varphi_p(s) \leq s^2 +1$ for all $s \in \mathbb{R}$ and $\mathbb{E}\int_D \pi_h[(u_h(x,t))^2] \dd x \leq C < \infty$ for all $t \in [0,T]$ (which follows from the energy estimate~\eqref{eq:est11} of Proposition~\ref{w-p-uh}); 
furthermore, $\lim_{p \rightarrow \infty} \varphi_p(s) = \frac{1}{2}([s]_{-})^2$ for all $s \in \mathbb{R}$. By passing to the limit $p \to +\infty$ we thus deduce from \eqref{eq-zero} that
\[ \mathbb{E}\int_D \pi_h[([u_h(x,t)]_{-})^2]\dd x = 0\quad \mbox{for all $t \in [0,T]$}. \]

Therefore, as a (nonnegative) random function, 
\[ \int_D \pi_h[([u_h(x,t)]_{-})^2]\dd x = 0\quad \mbox{almost surely, for all $t \in [0,T]$}.\]
Hence, $\pi_h[([u_h(x,t)]_{-})^2]\equiv 0$ on $\overline{D}$ almost surely for all $t \in [0,T]$, and therefore (owing to the linear independence of the finite element basis functions $\phi_1, \ldots, \phi_N$), also $[u_h(x_i,t)]_{-}=0$ almost surely at all mesh points $x_i$, $i \in \{1,\ldots,N\}$ and for all $t \in [0,T]$. Therefore, because $u_h(\cdot,t)$ is a continuous piecewise affine function for all $t \in [0,T]$, nonnegativity at all points in the mesh implies that $u_h(x,t) \geq 0$ for all $(x,t) \in \overline{D} \times [0,T]$. This concludes the proof of Theorem~\ref{thm:nonnegativity}.
\end{proof}

\begin{remark}
We mentioned at the beginning of Section~\ref{Sect:continousProblem} that variants of our original equation~\eqref{eq:contSPDE}$_1$ could be considered. The addition of a term $\,c\,u$ to the left-hand side of the equation with a real-valued function $c \in L^\infty(0,T; C(\overline{D}; \mathbb{R}_{\geq 0}))$ does not affect our results, and the proof presented above can be easily modified to include such a term. However, there is a technicality that we need to highlight. It should be noted that $\varphi_p$ has been constructed in such a way that we also have~$\varphi_p'(s) s \geq 0$ for all $s \in \mathbb{R}$. This allows us to show nonnegativity of the additional term
\[\int_0^t  \int_D\pi_h[\varphi_p'(u_h(x,s)) c(x,s) u_h(x,s)]\dd x \dd s,\]
which then appears on the left-hand side of It$\hat{\rm{o}}$'s formula~\eqref{eq:new-ito}, and which then, thanks to its nonnegativity, does not affect the proof of Theorem \ref{thm:nonnegativity}
presented above.
\end{remark}

\section{Error estimate}\label{Section:errorEstimates}

The purpose of this section is to establish our main convergence result (Theorem~\ref{thm:error} below). 
We begin by stating and proving three preliminary lemmas; we will then formulate and prove the main result of the section.

\subsection{Three lemmas regarding mass-lumping}

Our first technical lemma establishes the equivalence of the norms $\|\cdot\|_h$ (defined in~\eqref{eq:discrete-norm}) and the $L^2(D)$ norm $\|\cdot\|$ on $V_h$, with norm-equivalence constants that are independent of $h$.
To keep the exposition self-contained, we have included its proof.

\begin{lemma}\label{L2-h-norm-comparison}
There exists a positive constant $C$, independent of $h$, such that, for any function $v_h  \in V_h$ we have
     \begin{equation}\label{eq:17}
                \|v_h\|   \leq \|v_h\|_h \leq C \|v_h\|.  
    \end{equation}
In addition,  for any $v \in C_0(\overline{D})$, 
  \begin{equation}\label{eq:17bis}
  \|\pi_h v\| \leq \|v\|_h.
  \end{equation}
\end{lemma}
\begin{proof}
We begin by showing~\eqref{eq:17bis}, which will then, as a special case, imply the first inequality in \eqref{eq:17}.
Suppose that $K$ is a (closed) simplex in the triangulation $\mathcal{T}_h$ of $\overline{D}$, with vertices $x^K_j$, $j=0,\ldots,d$, and $v \in C_0(\overline{D})$. Then, for all~$x \in K$,
\begin{align*}
|\pi_hv(x)|^2 &=   \left| \sum_{j=0}^d v(x^K_j) \sqrt{\phi^K_j(x)} \sqrt{\phi^K_j(x)} \right|^2 \\
&\leq \left( \sum_{j=0}^d [v(x^K_j)]^2 \phi^K_j(x)\right) \left( \sum_{j=0}^d \phi^K_j(x)\right)  \\
&= \pi_h[v^2(x)].
\end{align*}
Hence, we have that 
\begin{align*}
\|\pi_h v\|^2   = \sum_{K \in \mathcal{T}_h}\int_K |\pi_h v(x)|^2 \dd x  \leq    \sum_{K \in \mathcal{T}_h}\int_K \pi_h [v^2(x)] \dd x =   \|v\|^2_h. 
\end{align*}
In particular, with $v= v_h \in {V}_h$, we have  
\[ \|v_h\|   = \|\pi_h v_h\|   \leq \|v_h\|_h, \]
which is the first inequality in \eqref{eq:17}. 

In order to prove the second inequality in \eqref{eq:17}, we first remark that 
$\|\pi_h [v_h^2]\|_{L^\infty(K)} \leq \|v_h\|^2_{L^\infty(K)}$. Indeed, 
\begin{align*}
\|\pi_h[v_h^2]\|_{L^\infty(K)} &= \max_{x \in K} \sum_{j=0}^d v_h^2(x_j^K) \phi_j^K(x)   \leq \max_{j \in \{0,1,\ldots,d\}} v_h^2(x_j^K)\max_{x \in K} \sum_{j=0}^d \phi_j^K(x)\\ &=  \max_{j \in \{0,1,\ldots,d\}} v_h^2(x_j^K) \times 1 \leq \|v_h\|^2_{L^\infty(K)}. 
\end{align*}
We therefore have 
\begin{align*}
\|v_h\|^2_h = \sum_{K \in \mathcal{T}_h} \int_K \pi_h[v_h^2(x)] \dd x & \leq \sum_{K \in \mathcal{T}_h} |K|\, \|\pi_h[v_h^2]\|_{L^\infty(K)} \leq \sum_{K \in \mathcal{T}_h} |K|\, \|v_h\|^2_{L^\infty(K)}\\
&\leq C\,\sum_{K \in \mathcal{T}_h} |K|\, |K|^{-1} \|v_h\|^2_{L^2(K)} = C\|v_h\|^2,
\end{align*}
where the last inequality holds thanks to the assumed shape-regularity of $\mathcal{T}_h$, which implies the existence of  a positive constant $C$ such that, for each closed  simplex $K \in \mathcal{T}_h$,
\[\|v_h\|_{L^\infty(K)} \leq C |K|^{-\frac{1}{2}}\|v_h\|_{L^2(K)}.\]
This completes the proof of the lemma. 
\end{proof}

The second lemma estimates the closeness of $(g,v_h)$ and $(g,v_h)_h$ for $g \in W^{2,\infty}(D)$ and $v_h \in V_h$.

\begin{lemma}\label{L2-h-scalar-product}
    There exists a constant $C$,  independent of $h= \max_{K \in \mathcal{T}_h} h_K$, 
    such that for any function $g \in W^{2,\infty}(D)$,  and for any $v_h \in V_h$, we have 
    \begin{align}\label{L2-h-comparison}
    |(g,v_h) - (g,v_h)_h| \leq C  h^2 \left(|g|_{W^{2,\infty}(D)}\|v_h\| + |g|_{W^{1,\infty}(D)}\|\nabla v_h\| \right).
    \end{align}
\end{lemma}

\begin{proof} We begin by observing that, because $ W^{2,\infty}(D)\hookrightarrow C(\overline{D})$, the inner product $(g,v)_h$ is well-defined. Next, because $v_h(x)=0$ for all $x \in \overline{D}\setminus D_h$, whereby also $(gv_h)(x)=0$ for all $x \in \overline{D}\setminus D_h$, we have that
\begin{align*} 
|(g,v_h) - (g,v_h)_h| & = \left|\int_{D} (gv_h)(x) - \pi_h(g v_h)(x) \dd x \right| \leq \int_{D_h} | (gv_h)(x) - \pi_h(g v_h)(x)| \dd x\\
& \leq   |D|^{\frac{1}{2}} \left(\sum_{K \in \mathcal{T}_h} \|(gv_h) - \pi_h(g v_h)\|^2_{L^2(K)} \right)^{\frac{1}{2}}\\
& \leq C h^2 \left(\sum_{K \in \mathcal{T}_h} |gv_h|^2_{H^2(K)} \right)^{\frac{1}{2}}\\
& \leq C  h^2 \left(\sum_{K \in \mathcal{T}_h}  \left(|g|^2_{W^{2,\infty}(D)}\|v_h\|^2_{L^2(K)} + |g|^2_{W^{1,\infty}(D)}\|\nabla v_h\|^2_{L^2(K)} \right)\right)^{\frac{1}{2}}\\
& = C  h^2 \left(|g|^2_{W^{2,\infty}(D)}\|v_h\|^2 + |g|^2_{W^{1,\infty}(D)}\|\nabla v_h\|^2 \right)^{\frac{1}{2}}.
\end{align*}
Here we have used a standard finite element interpolation error bound for piecewise affine elements, which asserts that on a family of shape-regular simplicial partitions, one has that $\|\varphi - \pi_h \varphi\|_{L^2(K)} \leq C h_K^2|\varphi|_{H^2(K)}$, $h_K:=\mbox{diam}(K)$, $C$ is a positive constant independent of $h= \max_{K \in \mathcal{T}_h} h_K$ and of $\varphi \in H^2(D)$ (continuously embedded in~$C(\overline D)$ for $d=1, 2, 3$); see, for example, Corollary 4.4.24 on p.110 in \cite{MR2373954}. Thus, we have shown \eqref{L2-h-comparison}.
\end{proof}

To state our third and final technical lemma, we first introduce the discrete $L^2$-projection $Q_h: L^2(D) \to V_h$ defined by
\begin{equation}\label{mixed-L2-proj}
    (Q_h w, v_h)_h := (w, v_h) \quad \forall\,  v_h \in V_h.
\end{equation}

A straightforward calculation shows that, for $w \in L^2(D)$, 
\[ (Q_h w)(x) = \sum_{j=1}^N W_j \phi_j(x), \quad \mbox{where} \quad 
W_j= \frac{\int_{\mathrm{supp}~\!\!\phi_j} w(y) \phi_j(y) \dd y}{\int_{\mathrm{supp}~\!\!\phi_j} \phi_j(y)\dd y}.\]
As $\mathrm{supp}~\!\phi_j \subset \overline{D_h}\subset \overline{D}$ and $\int_{\mathrm{supp}~\!\phi_j} \phi_j(y) \dd y >0$ for all $j=1,\ldots,N$, $W_j$ is correctly defined. This means that $Q_h$ is a quasi-interpolation operator. Then we have the following result. 
\begin{lemma}\label{lemma:g}
Suppose that $g \in W^{1,\infty}_0(\overline{D})$.  Then
\begin{equation}
\label{eq:lemmag}
 \| \pi_h g - Q_hg \|_h \leq  h \|\nabla g\|_{L^\infty(D)}.
 \end{equation}
\end{lemma}

\begin{proof} 
Expanding $\pi_h g \in V_h$ and $Q_h g \in V_h$ in terms of the basis $\{\phi_1,\ldots, \phi_N\}$ that spans $V_h$, we have that 
\begin{align*}
(\pi_h g) (x) - (Q_hg)(x) &= \sum_{j=1}^N g(x_j)\phi_j(x) - \sum_{j=1}^N \frac{\int_{\mathrm{supp}~\!\!\phi_j} g(y) \phi_j(y) \dd y}{\int_{\mathrm{supp}~\!\!\phi_j} \phi_j(y)\dd y} \phi_j(x)\\
&= \sum_{j=1}^N \frac{\int_{\mathrm{supp}~\!\!\phi_j} [g(x_j) - g(y)] \phi_j(y) \dd y}{\int_{\mathrm{supp}~\!\!\phi_j} \phi_j(y)\dd y} \phi_j(x),\quad x \in D.
\end{align*}
In particular, because $\mathrm{supp}~\!\phi_j \subset \overline{D_h}\subset \overline{D}$, we have that 
$(\pi_h g) (x) =0 $ and $ (Q_hg)(x) =0$ for all $x \in \overline{D}\setminus D_h$.
Therefore, 
\begin{align*}
(\pi_h g) (x_i) - (Q_hg)(x_i) = \frac{\int_{\mathrm{supp}~\!\!\phi_i} [g(x_i) - g(y)] \phi_i(y) \dd y}{\int_{\mathrm{supp}~\!\!\phi_i} \phi_i(y)\dd y}, \quad i=1,\ldots, N.
\end{align*}
This implies that 
\begin{align*}
|(\pi_h g) (x_i) - (Q_hg)(x_i)| & \leq  \frac{\int_{\mathrm{supp}~\!\!\phi_i} |g(x_i) - g(y)| \phi_i(y) \dd y}{\int_{\mathrm{supp}~\!\!\phi_i} \phi_i(y)\dd y}\\
&\leq \|\nabla g\|_{L^\infty(D)} \frac{\int_{\mathrm{supp}~\!\!\phi_i} |x_i - y| \phi_i(y) \dd y}{\int_{\mathrm{supp}~\!\!\phi_i} \phi_i(y)\dd y}\\ 
&\leq h \|\nabla g\|_{L^\infty(D)}\frac{\int_{\mathrm{supp}~\!\!\phi_i} \phi_i(y) \dd y}{\int_{\mathrm{supp}~\!\!\phi_i} \phi_i(y)\dd y} =  h \|\nabla g\|_{L^\infty(D)}, \quad i=1,\ldots,N,
\end{align*}
which concludes the proof. 
\end{proof}

We are now ready to embark on the proof of convergence of the numerical method. To this end, and as is customary, we need to make an additional assumption on the regularity of the exact solution $u$.
Unlike deterministic evolution problems, in the case of the stochastic partial differential equation under consideration here the solution can only be expected to possess smoothness with respect to the spatial variable $x \in D$, which is inherited from the regularity of the initial datum $u_0$, the smoothness of the function $f$, and the regularity of the functions $e_k$, $k=1,\ldots,M$, which appear in the definition of the finite-dimensional white noise $W^M$ stated in equation \eqref{noise_expansion}. To proceed, we shall therefore adopt the following regularity hypothesis.

\begin{assumption}[Regularity of the solution]\label{ass:reg}
    We shall assume in what follows that $u \in L^2(\Omega; L^2(0,T; H^2(D)\cap H^1_0(D)))$ 
    and $\Delta u \in L^2(\Omega; L^2(0,T;W^{2,\infty}(D)\cap W^{1,\infty}_0(D)))$.
\end{assumption}

When $D=\mathbb{T}^d$ (the $d$-dimensional torus) and instead of a homogeneous Dirichlet boundary condition on $\partial D \times (0,T]$ a periodic boundary condition is imposed on $u$, suitable assumptions on the data from which the regularity $u \in L^2(\Omega; L^2(0,T; C^4(\overline{D})))$ can be inferred are stated in Corollary 2.2 on p.762 of \cite{MR3057153}. 
If, on the other hand, $D$ is a bounded domain in $\mathbb{R}^d$ with a smooth boundary and a homogeneous Dirichlet boundary condition is imposed, the spatial regularity of $u$ follows from Theorem 2.7 on p.1598 in \cite{MR3340199}. In the case of polygonal domains in $\mathbb{R}^2$, conical domains in $\mathbb{R}^d$ and, more generally, Lipschitz domains in $\mathbb{R}^d$, the regularity theory of SPDEs is less complete; the only results that we are aware of are those stated in \cite{MR3004660,MR2875351,MR3109621,MR4001061,MR3757679,MR4483346}.

\begin{theorem}
\label{thm:error} Suppose that $u_0 \in C_0(\overline{D};\mathbb{R}_{\geq 0})$. Under the additional regularity Assumption \ref{ass:reg}, the following error estimate holds:
   \begin{align}\label{thm_error}
   \mathbb{E}(\|u(t) -u_h(t)\|^2) + 2 \int_0^t \mathbb{E}(\| \nabla (u -u_h)(s)\|^2) \dd s \leq C h^2 ,
\end{align}
where $C=C(T, u)$ is independent of the mesh size~$h$.
\end{theorem}

\begin{proof}
We proceed similarly as in the deterministic setting and consider the decomposition
\begin{equation}
\label{eq:decomposition}
    u - u_h = (u - \pi_h u) + (\pi_h u-u_h) =: \eta_h + \xi_h.
\end{equation}
As $\pi_h$ is idempotent on the linear space $C(\overline{D})$, it follows that $\pi_h \eta_h(t)=0$, and, equivalently, $(\eta_h(t),v_h)_h=0$ for all $v_h \in V_h$ and for all $t \in [0,T]$.
Next, note that with $u_h(0)=u_h^0 :=\pi_hu_0$ for $u_0 \in C(\overline{D})$ we have $\xi_h(0) = u_h(0) - \pi_h u(0) = 0$. 

The term~$\eta_h$ within~\eqref{eq:decomposition} will be dealt with (in Step~4 below) using a classical deterministic argument, so let us start by concentrating on the estimation of the term~$\xi_h$.
Using \eqref{FEM-formulation-h-product} and \eqref{eq:continous}, the definition of the discrete Laplacian \eqref{def:Delta_h-h} and the fact that $\pi_h$ is idempotent, we observe that the process $\xi_h$ in $(V_h, \|\cdot\|_h)$ satisfies the following equation 
\begin{align*}
    ( \xi_h(t), v_h)_h - \int_0^t ( \Delta_h \xi_h(s),  v_h)_h \dd s 
    &= [(u(t),v_h)_h - (u(t),v_h)] - [  (u^0_h, v_h)_h -  (u_0, v_h)]\\
    &\quad   + \int_0^t \sum_{k=1}^M (f(u(s))e_k), v_h) \dd B_k(s)\\
   & \quad -\int_0^t \sum_{k=1}^M (\pi_h(f(u_h(s))e_k), v_h)_h \dd B_k(s)\\
   &\quad - \int_0^t (\nabla \eta_h(s), \nabla v_h)\dd s \quad \forall\,  v_h \in V_h.
\end{align*}
In order to estimate~$ \|\xi_h(t)\|^2_h$ (which will be the purpose of Step 3 below), we will have to apply It$\hat{\rm{o}}$'s formula, which requires, in Step~1 of the argument below, to rewrite this equation in terms of the inner product $(\cdot,\cdot)_h$ (which will be of the form~\eqref{eq:first-step} below) and next, in Step~2, to separate the drift and diffusion terms to obtain the form~\eqref{eq:second-step} below, which then represents the starting point of our error analysis.

\smallskip 

\noindent {\bf Step 1:} 
Using the operator $Q_h$ introduced in~\eqref{mixed-L2-proj},  we have
\begin{align}\label{eq:xih}
\begin{aligned}
     ( \xi_h(t), v_h)_h &- \int_0^t (\Delta_h \xi_h(s), v_h)_h  \dd s = (\pi_h u(t) - Q_h u(t), v_h)_h - ( \pi_h u_0-Q_h u_0,v_h )_h \\
    &\quad + \int_0^t \sum_{k=1}^M (Q_h(f(u(s))e_k) - \pi_h(f(u_h(s))e_k), v_h)_h  \dd B_k(s) \\
     &\quad - \int_0^t (\nabla \eta_h(s), \nabla v_h) \dd s\quad \forall\,  v_h \in V_h.
\end{aligned}
\end{align}
The last term involves the standard $L^2(D)$ inner product rather than the discrete inner product $(\cdot,\cdot)_h$. We shall therefore also express this last term in terms of the inner product $(\cdot,\cdot)_h$.
To this end, we let $\Pi_h: L^2(D) \to V_h$ denote the standard $L^2$-projector, defined by $(u,v_h) = (\Pi_h u, v_h)$
for all $v_h \in V_h$. 
Next, following \cite{MR1921920} (see also \cite{MR3154916}), we introduce an extension~$ \tilde{\Pi}_h$ of the $L^2$-projector $\Pi_h$ from $L^2(D)$ to the function space $H^{-1}(D):=[H^1_0(D)]^* \hookleftarrow [L^2(D)]^* \equiv L^2(D)\hookleftarrow H^1_0(D)$,  with values in~$V_h$ and defined as follows: 
\begin{equation*}
(\tilde{\Pi}_h w, v_h) = \langle w,v_h\rangle_{H^{-1}(D),H^1_0(D)} = (\nabla((-\Delta)^{-1}w),\nabla v_h) \quad \forall\,  (w,v_h) \in H^{-1}(D) \times V_h, 
\end{equation*}
where (the bounded linear operator) $(-\Delta)^{-1}: H^{-1}(D) \to H^1_0(D)$  denotes the inverse of the Dirichlet Laplacian $-\Delta: H^1_0(D) \to H^{-1}(D)$, which is clearly a bijection (e.g., by a straightforward application of the Lax--Milgram lemma).

Taking $w = -\Delta \eta_h \in H^{-1}(D)$ in the definition of $\tilde{\Pi}_h$ (which is possible because $\eta_h = u - \pi_h u \in H^1_0(D)$ and the Dirichlet Laplacian $\Delta$ maps $H^1_0(D)$ onto $H^{-1}(D)$), we have
\begin{align*}
     \int_0^t (\nabla \eta_h(s), \nabla v_h) \dd s &=  \int_0^t \langle - \Delta \eta_h(s), v_h \rangle_{H^{-1}(D), H^1_0(D)} \dd s \\
     &=  \int_0^t ( \tilde{\Pi}_h (-\Delta \eta_h(s)), v_h) \dd s \\
     &=  \int_0^t - (Q_h \tilde{\Pi}_h \Delta \eta_h, v_h)_h\quad \forall\,  v_h \in V_h.
\end{align*}

All the terms on the right-hand side of~\eqref{eq:xih} are now written in terms of the discrete inner product $(\cdot,\cdot)_h$. Hence, 
\begin{align}
\label{eq:first-step}
\xi_h(t) &- \int_0^t \Delta_h \xi_h(s) \dd s 
= \pi_h(u(t) - u_0) - Q_h(u(t) - u_0)\nonumber\\
    &+ \int_0^t \sum_{k=1}^M (Q_h(f(u(s))e_k) - \pi_h(f(u_h(s))e_k)) \dd B_k(s) + \int_0^t Q_h \tilde{\Pi}_h \Delta \eta_h(s) \dd s
\end{align}
as an equality in $\left(V_h,(\cdot,\cdot)_h\right)$. 

\smallskip 

\noindent {\bf Step 2:}
As was indicated above, we now get to our second modification, where we separate the drift and the diffusion terms within~\eqref{eq:first-step} (actually in the first two terms of the right-hand side of~\eqref{eq:first-step}). This is the step where the regularity Assumption \ref{ass:reg} is useful. 

We start by applying the linear and bounded mappings $\pi_h$ and $Q_h$ to \eqref{eq:contSPDE}. We note that because $u(\cdot,t)|_{\partial D} = 0$ for all $t \in [0,T]$, $f(0)=0$ and $e_k \in C(\overline{D})$ for $k=1,\ldots,M$, it follows that $\Delta u(\cdot,t)|_{\partial D}=0$ for all $t \in (0,T]$. We thus have
\begin{equation*}
    \pi_h (u(t) - u(0)) = \int_0^t {\pi}_h (\Delta u(s)) \dd s + \int_0^t \sum_{k=1}^M \pi_h (f(u)e_k) \dd B_k(s)
\end{equation*}
and 
\begin{equation*}
    Q_h (u(t) - u(0)) = \int_0^t {Q}_h (\Delta u(s)) \dd s + \int_0^t \sum_{k=1}^M Q_h (f(u)e_k) \dd B_k(s).
\end{equation*}
Therefore, 
\begin{align}
\label{eq:second-step}
&\xi_h(t) - \int_0^t \Delta_h \xi_h(s) \dd s = \int_0^t \big ({\pi}_h (\Delta u(s)) - {Q}_h (\Delta u(s))\big) \dd s \nonumber\\
&\quad + \int_0^t \sum_{k=1}^M \big(\pi_h (f(u(s))e_k)  - Q_h (f(u)e_k)\big) \dd B_k(s)\nonumber\\
    &\quad + \int_0^t \sum_{k=1}^M (Q_h(f(u(s))e_k) - \pi_h(f(u_h(s))e_k)) \dd B_k(s)  +  \int_0^t Q_h \tilde{\Pi}_h \Delta \eta_h(s) \dd s.\nonumber\\
&= \int_0^t \big ({\pi}_h (\Delta u(s)) -{Q}_h (\Delta u(s))\big) \dd s \nonumber\\
    &\quad + \int_0^t \sum_{k=1}^M \pi_h\big(f(u(s))e_k - f(u_h(s))e_k\big) \dd B_k(s) + \int_0^t Q_h \tilde{\Pi}_h \Delta \eta_h(s) \dd s.
\end{align}

\smallskip

\noindent {\bf Step 3:} We now intend to estimate~$ \|\xi_h(t)\|^2_h$. In order to apply It$\hat{\rm{o}}$'s formula in our semi-discrete setting, we consider the Gelfand triple $\mathcal{V}_h \hookrightarrow \mathcal{H}_h \hookrightarrow \mathcal{V}_h^*$, where now $\mathcal{V}_h$ denotes the linear space $V_h\subset H^1_0(D)$ equipped with the inner product $(\nabla \cdot\,,\, \nabla \cdot)$ understood as a Hilbert space, while $\mathcal{H}_h$ signifies $V_h$ equipped with the $L^2(D)$ discrete inner product $(\cdot,\cdot)_h$ understood as a Hilbert space into which $\mathcal{V}_h$ is continuously (and, trivially,) densely embedded. As $V_h$ is finite-dimensional, the topological dual, $\mathcal{V}_h^*$ of $\mathcal{V}_h$ is identified with $\mathcal{V}_h$. Therefore, while in the algebraic sense the equality sign in \eqref{eq:second-step} is just an equality between elements in $V_h$ (for each $t \in (0,T]$), in the topological sense it can
be understood as an equality in the dual space $\mathcal{V}_h^*$ featuring in the above Gelfand triple
(for each $t \in (0,T]$). Thanks to It$\hat{\rm o}$'s formula (cf. \cite{Krylov,MR651582})  applied on this Gelfand triple to $\|\xi_h\|_h^2$, we obtain
\begin{align}
\label{eq:T1234}
&    \|\xi_h(t)\|^2_h + 2 \int_0^t \| \nabla \xi_h(s)\|^2 \dd s \nonumber\\&= 2 \int_0^t \left\langle \pi_h(\Delta u(s)) - Q_h (\Delta u(s)),\xi_h(s)\right\rangle_{\mathcal{V}_h^*,\mathcal{V}_h}\!\!\dd s \nonumber\\
    &\quad + 2 \int_0^t \langle Q_h \tilde{\Pi}_h(\Delta \eta_h(s)),\xi_h(s)\rangle_{\mathcal{V}_h^*, \mathcal{V}_h} \!\dd s+ \int_0^t \sum_{k=1}^M\|\pi_h(f(u(s))e_k - f(u_h(s))e_k)\|^2_h \dd s + M_t \nonumber\\
    & =: \mathrm{T}_1 + \mathrm{T}_2 + \mathrm{T}_3 + M_t,
\end{align}
where $M_t$ a martingale term, which follows directly from the energy estimates \eqref{energy_estimate_u} and \eqref{eq:est11}. We now successively bound all terms.


From Assumption~\ref{ass:reg}, we have 
\begin{align*}
    \mathrm{T}_1&:=  2 \int_0^t  \langle  {\pi}_h (\Delta u(s))  - {Q}_h (\Delta u(s)), \xi_h(s) \rangle_{\mathcal{V}_h^*,\mathcal{V}_h}\dd  s 
        \\
&= 2 \int_0^t  \left( {\pi}_h (\Delta u(s))  - {Q}_h  (\Delta u(s)), \xi_h(s) \right)_h\dd s = 2 \int_0^t  \left( \Delta u(s), \xi_h(s) \right)_h   - \left(   \Delta u(s), \xi_h(s) \right)\dd s,
\end{align*}
where we have used that $\Delta u \in C(\overline{D}) \subset L^2(D)$. In order to bound this term, we use Lemma \ref{L2-h-scalar-product} with $g=\Delta u \in W^{2,\infty}(D)$ and~\eqref{eq:17} in Lemma \ref{L2-h-norm-comparison}. As a result we obtain 
\begin{align}
    T_1 &\leq 2 C h^2 \int_0^t\left(|\Delta u(s)|_{W^{2,\infty}(D)}\|\xi_h(s)\| + |\Delta u(s)|_{W^{1,\infty}(D)}\|\nabla \xi_h(s)\| \right)\dd s \nonumber\\
    &\leq C^2 h^4 \int_0^t \left(|\Delta u(s)|_{W^{2,\infty}(D)}^2 +|\Delta u(s)|_{W^{1,\infty}(D)}^2 \right)\dd s +  \frac{1}{2}\int_0^t \left(\|\xi_h(s)\|^2_h +\| \nabla \xi_h(s)\|^2 \right) \! \dd s.
\end{align}

The term $\mathrm{T}_2$ is bounded from above as follows
\begin{align*}
    \mathrm{T}_2&:= 2 \int_0^t \langle Q_h \tilde{\Pi}_h(\Delta \eta_h(s)),\xi_h(s)\rangle_{\mathcal{V}_h^*, \mathcal{V}_h} \dd s =  2 \int_0^t \left( Q_h \tilde{\Pi}_h(\Delta \eta_h(s)),\xi_h(s)\right)_{\mathcal{H}_h}  \dd s\\
    & =  2 \int_0^t \left( Q_h \tilde{\Pi}_h(\Delta \eta_h(s)),\xi_h(s)\right)_h  \dd s =  2 \int_0^t \left(\tilde{\Pi}_h(\Delta \eta_h(s)),\xi_h(s)\right)  \dd s\\
    &= -2 \int_0^t (\nabla \eta_h(s), \nabla \xi_h(s)) \dd s\\
    & \leq 2\int_0^t \|\nabla \eta_h(s)\|^2 \dd s + \frac{1}{2}\int_0^t \|\nabla \xi_h(s)\|^2 \dd s\\
    &\leq  Ch^2 \|u\|^2_{L^2(0,T;H^2(D))} + \frac{1}{2}\int_0^t \|\nabla \xi_h(s)\|^2 \dd s,
\end{align*}
where we have used the interpolation error bound $\|\nabla \eta_h\| = \|\nabla(u - \pi_h u)\| \leq C h \|u\|_{H^2(D)}$ (see, for example, inequality (19) on p.86 of  \cite{MR584442}).

For the term $\mathrm{T}_3$, using the nonnegativity of the continuous piecewise affine finite element basis functions together with the properties of $f$ and $e_k$, $k=1,\ldots,M$, we have that 
\begin{align*}
    \mathrm{T}_3& \leq M c_e^2 c_f^2 \int_0^t \|u(s) - u_h(s)\|^2_h \dd s = M c_e^2 c_f^2 \int_0^t \|\xi_h(s)\|^2_h \dd s.
\end{align*}
Here, to deduce the last equality, we made use of the fact that $\eta(x,s)=u(x,s) - (\pi_hu)(x,s) = 0$ at each vertex $x$ of the triangulation for all $s \in [0,T]$; hence, 
$(u-u_h)(x,s) = \xi_h(x,s)$ at each vertex of the triangulation, whereby $\|u(s)-u_h(s)\|_h = \|\xi_h(s)\|_h$ for all $s \in [0,T]$.

Having bounded the terms $\mathrm{T}_1$, $\mathrm{T}_2$ and  $\mathrm{T}_3$, we insert these bounds into~\eqref{eq:T1234} and take the expectation of the resulting inequality. The martingale term then vanishes, and it follows from Gr\"onwall's inequality that
\begin{align*}
   \mathbb{E}(\|\xi_h(t)\|^2_h) + 2 \int_0^t \mathbb{E}(\| \nabla \xi_h(s)\|^2) \dd s \leq 
   C h^2 \tilde{C}(u)
\end{align*}
where $\tilde{C}(u) := \mathbb{E}
   \left(\|\Delta u(s)\|_{L^2(0,T;W^{2,\infty}(D))}^2 + \|u\|^2_{L^2(0,T:H^2(D))} \right)$.
Thanks to Lemma \ref{L2-h-norm-comparison}, also
\begin{align*}
   \mathbb{E}(\|\xi_h(t)\|^2) + 2 \int_0^t \mathbb{E}(\| \nabla \xi_h(s)\|^2) \dd s \leq C h^2, 
\end{align*}
where $C=C(T,u)$ is independent of the mesh size.

\smallskip

\noindent {\bf Step 4:}
By a standard finite element interpolation error bound for piecewise affine elements on a shape-regular triangulation (cf., again, Corollary 4.4.24 on p.110 in \cite{MR2373954} in the case when $D=D_h$ is a polygonal domain in $\mathbb{R}^2$ or a Lipschitz polyhedron in $\mathbb{R}^3$, and inequality (18) on p.86 of  \cite{MR584442} in the case when $D$ has $C^2$ boundary), we have, thanks to Assumption \ref{ass:reg}, that $\|\eta_h(t)\| = \|u(t) - \pi_h u(t) \| \leq C h^2|u(t)|_{H^2(D)}$, where $C$ is a positive constant independent of $h$.  It remains to note that $u(t)-u_h(t) = \eta_h(t) + \xi_h(t)$ and apply the triangle inequality to deduce the desired bound \eqref{thm_error} on the discretization error $u(t)-u_h(t)$ for the mass-lumped finite element method.
\end{proof}

\section{Fully discrete schemes and numerical simulations}
\label{Section:NumericalExperiments}

The purpose of this final section is to demonstrate that the theoretical analysis of the continuous-in-time mass-lumped finite element method presented in the previous sections, which we believe to be of interest in itself, also opens a useful \emph{practical} perspective. 

In order to support this claim, we show that, at least in the case of $d=2$, our continuous-in-time finite element method may be complemented by a suitable time-integrator so that the resulting fully discrete scheme is both accurate and nonnegativity-preserving.  
We restrict our attention to the case $f(u) = \lambda u$ and assume that we only have a single Brownian motion, i.e., $M=1$ in~\eqref{noise_expansion}, so that we simply denote $e_1(x)$ by $e(x)$. 
Extensions of this setting may of course be considered. The extension to the case of $M\geq 2$ Brownian motions is relatively straightforward. On the other hand, considering a nonlinear right-hand side $f(u)$ requires more work. A simple idea, which is economical to implement, given the schemes we introduce below, could be an additional approximation using a linearization of $f$ over each time step, in the spirit of what was done in~\cite{brehier2023analysis}. We will not comment on this further as we do not, by any means, aim at generality in this illustrative final section.

\subsection{Fully discrete schemes}

In what follows, we construct a fully discrete scheme for \eqref{eq:contSPDE}, based on our semi-discrete scheme~\eqref{SDE_formulation}. 
Let $\Delta t$ denote the time step and let $G_n$ denote independent draws of a normalized Gaussian variable. 
Let $\tilde{M}$ be the diagonal matrix with entries $e(x_i)$ and $A$ the matrix defined in~\eqref{eq:matrixA}. We consider the scheme given by
\begin{equation}
\label{eq:fully-discrete}
U_{n+1}\, + \,\Delta t \,A\,U_{n+1} = \exp \left(\lambda G_n\, \sqrt{\Delta t}\, \tilde{M} - \frac{1}{2} \lambda^2\, \Delta t\, \tilde{M}^{2} \right) U_n,
\end{equation}
which may also be written as the two-step splitting scheme
\begin{equation}
\label{eq:fully-discrete-stepping}
  \left\{
  \begin{array}{l}
  U_{n+1/2}\,=\, \exp \left(\lambda G_n \,\sqrt{\Delta t}\, \tilde{M} - \frac{1}{2} \lambda^2\, \Delta t \, \tilde{M}^{2} \right) U_n,\\
    U_{n+1}\, + \,\Delta t \,A\,U_{n+1} \,= \,  U_{n+1/2}.
     \end{array}
\right.
\end{equation}
It is immediate to see from~\eqref{eq:fully-discrete-stepping} that, by construction, the scheme preserves nonnegativity. Because we are using a triangulation that is weakly acute, so that our mass-lumped finite element method has this property in the deterministic setting, $U_{n+1}$ is nonnegative as soon as~$U_{n+1/2}$ is.  On the other hand, because the exponential prefactor in \eqref{eq:fully-discrete-stepping}$_1$ is the exponential of a diagonal matrix, whereby it is a diagonal matrix with positive entries, $U_{n+1/2}$ and $U_{n}$ share the same sign. Combining the two steps, we obtain preservation of nonnegativity, irrespective of the choice of the time step~$\Delta t$.

Formally, the time-stepping scheme defined by \eqref{eq:fully-discrete}, or~\eqref{eq:fully-discrete-stepping}, is an exponential form of the Euler--Milstein scheme. 
Given our focus on nonnegativity-preservation, using such an exponential form is a classical idea. It is for instance very close to the integrator employed in the recent work~\cite{brehier2023analysis}, as a local-in-time linearization followed by a spatial finite difference discretization.

It is also clear, at least intuitively, that the scheme has the same order of accuracy as the Euler--Milstein scheme, and that it therefore has strong order of convergence $1$ in time. This can be rigorously proved, but we will not proceed in that direction. Our numerical experiments below will confirm empirically that this is indeed the case.

An improvement of this scheme is obtained as follows. It is classical in the numerical discretization of \emph{ordinary} (as opposed to stochastic) differential equations as well as of time-dependent  \emph{deterministic} (as again opposed to stochastic) partial differential equations, to consider the symmetric, or Strang, three-step version of a two step splitting. This is known to improve the order of discretization by one. The same improvement is observed, in terms of the \emph{weak} order of convergence, for stochastic differential equations. For the SPDE under consideration in this work, it is therefore natural, precisely with a view to improving the accuracy of the discretization in time (while still ensuring nonnegativity-preservation, a fact that is obvious from the following formulae), to change~\eqref{eq:fully-discrete-stepping} into the following three-step splitting scheme:
\begin{equation}
\label{eq:fully-discrete-stepping-2}
  \left\{
  \begin{array}{l}
  \displaystyle  U_{n+1/3}\, + \,\frac{\Delta t}{2} \,A\,U_{n+1/3} \,= \,  U_{n},\\
   U_{n+2/3}\,=\, \exp \left(\lambda G_n\, \sqrt{\Delta t} \, \tilde{M} - \frac{1}{2} \lambda^2\, \Delta t \,\tilde{M}^{2} \right) U_{n+1/3},\\
     \displaystyle  U_{n+1}\, + \,\frac{\Delta t}{2} \,A\,U_{n+1} \,= \,  U_{n+2/3}.
    \end{array}
\right.
\end{equation}
Alternatively, we could consider
\begin{equation}
\label{eq:fully-discrete-stepping-2bis}
  \left\{
  \begin{array}{l}
   \displaystyle  U_{n+1/3}\,=\, \exp \left(\lambda G^1_n\, \sqrt{\frac{\Delta t}{2}} \, \tilde{M} - \frac{1}{2} \lambda^2\, \frac{\Delta t}{2} \,\tilde{M}^{2} \right) U_{n},\\
  \displaystyle  U_{n+2/3}\, + \,{\Delta t} \,A\,U_{n+2/3} \,= \,  U_{n+1/3},\\
   \displaystyle  U_{n+1}\,=\, \exp \left(\lambda G^2_n\, \sqrt{\frac{\Delta t}{2}} \, \tilde{M} - \frac{1}{2} \lambda^2\, \frac{\Delta t}{2} \,\tilde{M}^{2} \right) U_{n+2/3},\\
     \end{array}
\right.
\end{equation}
where $G^1_n$ and $G^2_n$ denote independent draws of a normalized Gaussian variable.
The next subsection illustrates by numerical experiments the use of \eqref{eq:fully-discrete-stepping}, \eqref{eq:fully-discrete-stepping-2}, \eqref{eq:fully-discrete-stepping-2bis}, providing some comparisons of these schemes as well as with other alternative schemes that could be considered. 

Our experiments will show that \eqref{eq:fully-discrete-stepping} is the best option, along with its improvement \eqref{eq:fully-discrete-stepping-2}. It is as accurate (if not more accurate than) classical schemes and, in addition, it unconditionally preserves nonnegativity.

\subsection{Numerical experiments}

We consider the two-dimensional domain $D :=(0,1) \times (0,1)$. We use continuous piecewise 
affine polynomials on a uniform triangulation $\mathcal{T}_h$ of $\overline{D}$ with spacing $h>0$ in both coordinate directions.
The parameters chosen for our tests are:
$150$ realizations (of Brownian paths) to compute the expectation,
$u_0(x,y) := \sin(\pi\,x)\,\sin(\pi\,y)$ as initial condition and  $e(x,y) = \sin(\pi\,x)\,\sin(\pi\,y)$ as spatial noise factor.  

\paragraph*{Accuracy}

In our first set of numerical experiments, we assess the rate of convergence of the scheme~\eqref{eq:fully-discrete-stepping} with respect to both the spatial discretization paramater~$h$ and the time step~$\Delta t$. We set the endpoint of the time interval to be~$T = 1/2$ and $\lambda=3$. For the spatial discretization, we always consider our mass-lumped finite element method given by \eqref{ML-FEM_formulation}, or equivalently~\eqref{SDE_formulation}. We compare~\eqref{eq:fully-discrete-stepping} with two other time discretization methods, namely the Euler--Maruyama (EMa) and the Euler--Milstein (EMi) schemes. 
The reference solution $u_{ref}$ is computed using $\Delta t=2^{-14}$ and $h = 2^{-6}$. We shall focus our attention, in all our experiments below, on the strong order of convergence of the schemes in the probabilistic sense. We therefore always use the same noise samples for all time integrators, so as to be able to evaluate the \emph{strong} error. This suffices for our specific purposes here. The total (strong in the probability sense) error is defined by
\begin{equation}
\label{eq:definition-error}
    \sup_{0\leq n \leq K} \mathbb{E}(\| u_{h,K}(n\Delta t) - u_{ref}(n \Delta t)\|^2) +
    \int_0^T \mathbb{E} (\| \nabla u_{h,K}(s) - \nabla u_{ref}(s)\|^2) \dd s
\end{equation}
with $K=T/(\Delta t)$.  
We compute the time integral with the composite trapezium rule based on the quadrature points $n \Delta t$, $n=0,1,\ldots,K$. 

Figure \ref{fig:loglog_total_error} displays the associated error on a log-log scale. Note that~\eqref{eq:definition-error} is the \emph{square} of the strong numerical error, so all rates observed on the graphs are \emph{squared} orders of convergence. 
As expected, since the spatial discretization is based on piecewise affine finite elements, it is observed on the right of Figure \ref{fig:loglog_total_error} that all fully discrete methods considered, which we recall only differ in terms of the choice of the time-stepping scheme, share the same order of convergence in space; namely, they all exhibit first order convergence with respect to $h$. As far as convergence in time is concerned, we observe first order convergence for EMi, and order of only $1/2$ for EMa, both as predicted by theory. Our  scheme \eqref{eq:fully-discrete-stepping}  shares the first order convergence of the EMi scheme.

\begin{figure}
\centering
\begin{subfigure}{.5\textwidth}
  \centering
  \includegraphics[width=1\linewidth]{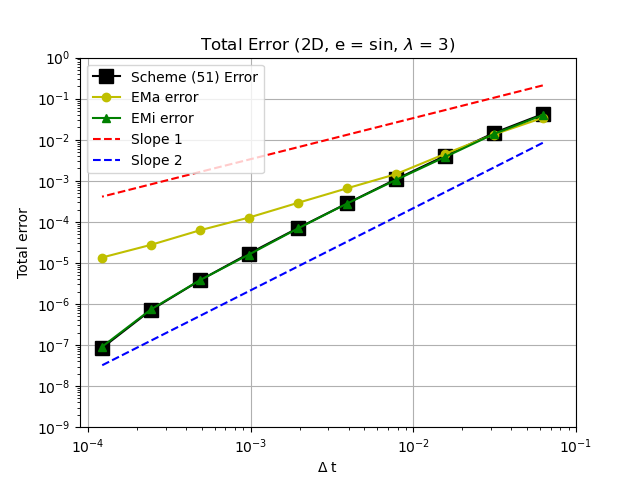}
  \caption{Fixed spatial step $h=2^{-6}$}
  \label{fig:sub1a}
\end{subfigure}%
\begin{subfigure}{.5\textwidth}
  \centering
  \includegraphics[width=1\linewidth]{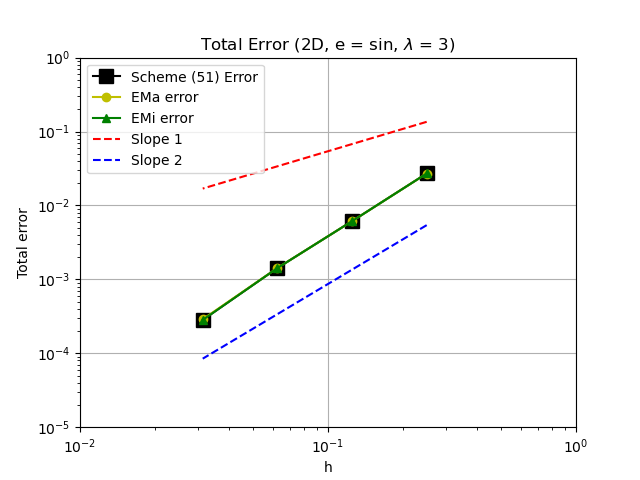}
  \caption{Fixed time step $\Delta t=2^{-14}$ }
  \label{fig:sub2a}
\end{subfigure}
\caption{Rates of convergence for the error as defined in~\eqref{eq:definition-error} (note that~\eqref{eq:definition-error} is the \emph{square} of the strong numerical error) of the two-step splitting scheme~\eqref{eq:fully-discrete-stepping}, at fixed mesh size~$h=2^{-6}$ when the time step~$\Delta t$ varies  (left), and at fixed time step~$\Delta t=2^{-14}$ when~$h$ varies  (right). Comparisons are shown with the results obtained using the Euler--Maruyama (EMa) and Euler--Milstein (EMi) schemes. Both graphs use a log-log scale. 
}
\label{fig:loglog_total_error}
\end{figure}

To complement our tests regarding rates of convergence, we also plot, in Figure \ref{fig:total_error}, the actual numerical error. We use the same parameters as in Figure \ref{fig:loglog_total_error} and again consider both cases, i.e., fixed spatial mesh size and  fixed time step. In all graphs of Figures \ref{fig:loglog_total_error} and  \ref{fig:total_error}, the errors of the schemes~\eqref{eq:fully-discrete-stepping} and EMi are visually at least (almost) indistinguishable. Put differently, the exponential form of our scheme~\eqref{eq:fully-discrete-stepping} does not modify the order of convergence of the classical EMi scheme. The difference between~\eqref{eq:fully-discrete-stepping} and all of the other schemes considered is however much more dramatic when it comes to nonnegativity-preservation.

\begin{figure}[ht]
\centering
\begin{subfigure}{.5\textwidth}
  \centering
  \includegraphics[width=1\linewidth]{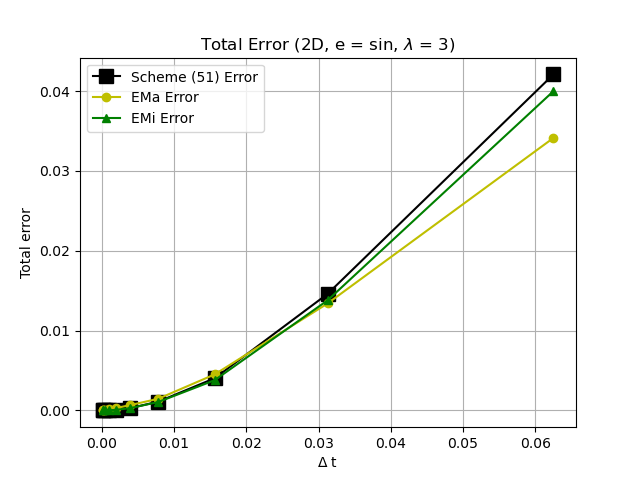}
  \caption{Fixed spatial step $h=2^{-6}$}
  \label{fig:sub1b}
\end{subfigure}%
\begin{subfigure}{.5\textwidth}
  \centering
  \includegraphics[width=1\linewidth]{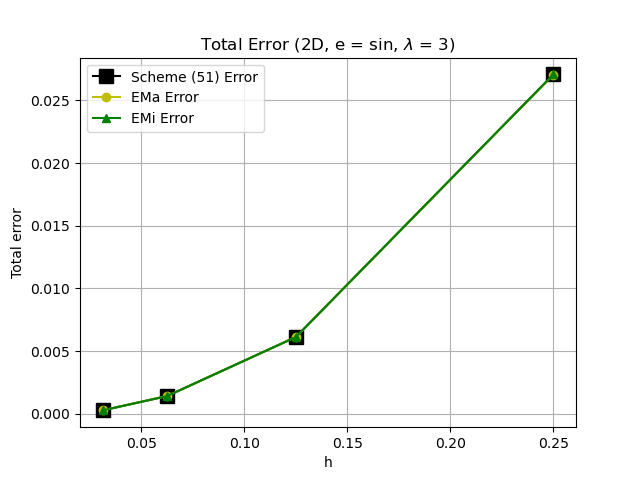}
  \caption{Fixed time step $\Delta t=2^{-14}$ }
  \label{fig:sub2b}
\end{subfigure}
\caption{Numerical errors, on a linear scale, for the same schemes and with the same parameters as those in Figure~\ref{fig:loglog_total_error} above.}
\label{fig:total_error}
\end{figure}

\begin{figure}[ht]
\centering
\begin{subfigure}{.5\textwidth}
  \centering
  \includegraphics[width=0.95\linewidth]{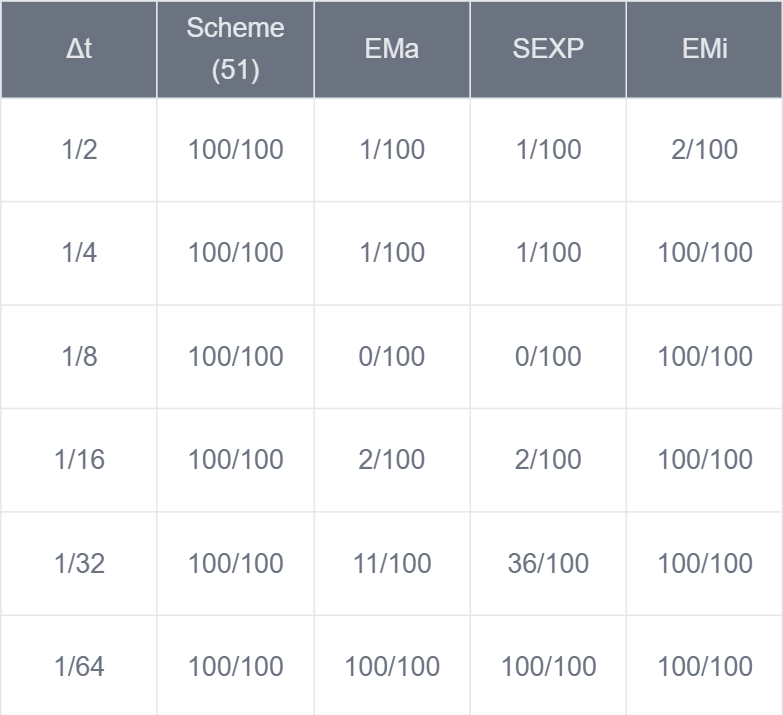}
  \caption{$\lambda =2$}
  \label{fig:sub1c}
\end{subfigure}%
\begin{subfigure}{.5\textwidth}
  \centering
  \includegraphics[width=0.95\linewidth]{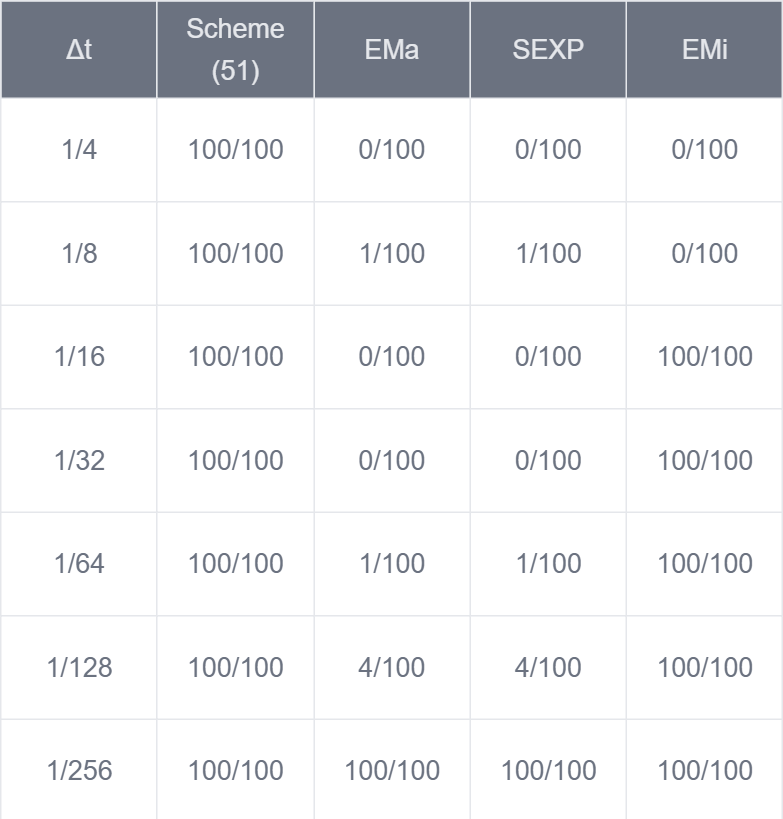}
  \caption{$\lambda=4$ }
  \label{fig:sub2c}
\end{subfigure}
\caption{Nonnegativity-preservation: comparison, for two values of the coefficient~$\lambda=2$ and $\lambda=4$ for the linear right-hand side $f(u)=\lambda\,u$, of the two-step splitting scheme~\eqref{eq:fully-discrete-stepping} with three other schemes: Euler--Maruyama (EMa), Euler Milstein (EMi) and the stochastic exponential Euler integrator scheme (SEXP) from~\cite{lord2013stochastic}. 
}
\label{fig:nonnegativity}
\end{figure}

\paragraph*{Nonnegativity-preservation}
In order to illustrate the issue of nonnegativity-preservation, we add another scheme to our portfolio, namely the \emph{stochastic exponential Euler integrator scheme} (SEXP) introduced in \cite{lord2013stochastic}. All of our experiments are performed in the same way as above, except that we consider a longer time interval with end time $T=2$, which makes the computations more demanding; we focus on our main motivational issue: preservation of nonnegativity. We fix once and for all the mesh size $h=2^{-4}$ and vary the time step.
Our numerical results are presented in Figure \ref{fig:nonnegativity}, where each entry  $k/100$ 
 is the number of solutions out of 100 simulations that remain positive over the entire time interval. The results confirm that the sequence of numerical approximations generated by \eqref{eq:fully-discrete-stepping} always remains positive and that this is not the case for the other time integrators. This is of course expected since \eqref{eq:fully-discrete-stepping} unconditionally preserves nonnegativity by construction. When we progressively reduce the time step, in order to let the competitors of  \eqref{eq:fully-discrete-stepping} achieve nonnegativity-preservation as a byproduct of increased precision, we observe that the results of these schemes improve. The results with these other schemes depend on the choice of the time step compared to the strength of the noise determined by the parameter $\lambda$.  If $ |\lambda| \sqrt{\Delta t}  \|e\|_{\infty} \leq 1$, then EMi always maintains nonnegativity. On the other hand, both EMa and SEXP generate numerical solutions which occasionally become negative, although if $|\lambda|  \sqrt{\Delta t}  \|e\|_{\infty}$ is small, then there is a much higher probability of maintaining nonnegativity.  In any event, reducing the time step is computationally expensive. The superiority of the scheme  \eqref{eq:fully-discrete-stepping} in this respect is therefore clear.

 Note that we have also compared \eqref{eq:fully-discrete-stepping} with a strategy commonly used by practitioners (and often referred to as \emph{clipping}), namely the plain Euler--Milstein method where the result is truncated when the method returns negative values. 
 For large values of lambda, the Euler--Milstein method with clipping becomes inaccurate, while \eqref{eq:fully-discrete-stepping} still performs correctly, both in terms of accuracy and, of course, non-negativity.

\paragraph*{Three-step splitting} 

Our last experiment is meant to demonstrate that \eqref{eq:fully-discrete-stepping-2} improves 
on \eqref{eq:fully-discrete-stepping} in terms of accuracy, while of course (and we do not repeat any dedicated tests) it exhibits nonnegativity-preservation by construction. We recall that this improvement is not unexpected, based on what is theoretically known in the realms of ODEs, PDEs, and SDEs. This could also be studied theoretically for our SPDEs, but we prefer to concentrate on more practical considerations here.

Figure \ref{fig:strang} (graph on the left) shows, compared to the graph on the left of Figure~\ref{fig:loglog_total_error},  that the \emph{strong} numerical error for \eqref{eq:fully-discrete-stepping-2} has essentially the same \emph{rate} of convergence as~\eqref{eq:fully-discrete-stepping}, but that the \emph{actual size} of this error is  smaller than that for  \eqref{eq:fully-discrete-stepping}. The improvement obtained for  \eqref{eq:fully-discrete-stepping-2bis} (graph on the right of Figure \ref{fig:strang}) is less evident. 
Based on the considerations developed above for ODEs, PDEs and SDEs, we presume that the \emph{weak} order of convergence will be increased by one when passing from~\eqref{eq:fully-discrete-stepping} to \eqref{eq:fully-discrete-stepping-2} and \eqref{eq:fully-discrete-stepping-2bis}. We do not conduct here this complementary set of tests regarding weak convergence. The results presented here suffice for us to be able to advocate \eqref{eq:fully-discrete-stepping-2} as the best option. 

\begin{figure}
\centering
\begin{subfigure}{.5\textwidth}
  \centering
  \includegraphics[width=1\linewidth]{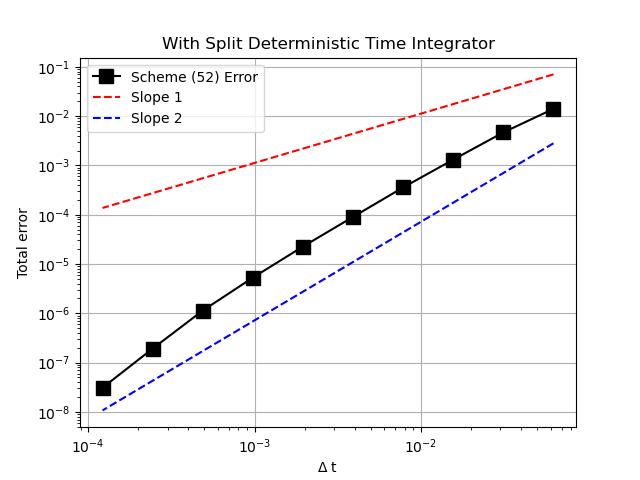}
  \caption{Strang scheme~\eqref{eq:fully-discrete-stepping-2}}
  \label{fig:sub1d}
\end{subfigure}%
\begin{subfigure}{.5\textwidth}
  \centering
  \includegraphics[width=1\linewidth]{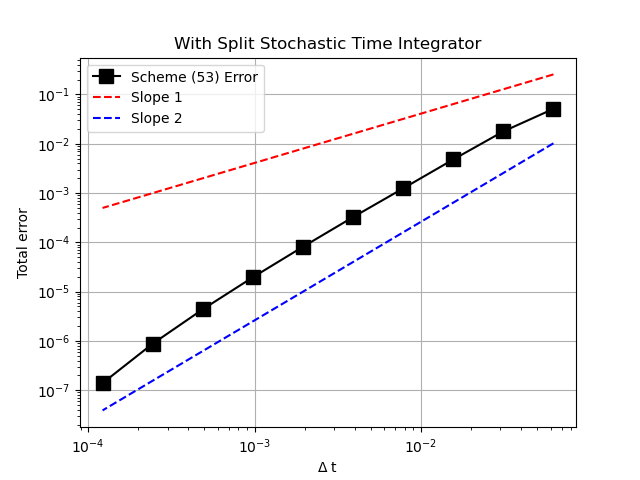}
  \caption{Strang scheme~\eqref{eq:fully-discrete-stepping-2bis}}
  \label{fig:sub2d}
\end{subfigure}
\caption{Numerical error of the Strang splitting schemes~\eqref{eq:fully-discrete-stepping-2}, on the left, and~\eqref{eq:fully-discrete-stepping-2bis}  on the right, both in Log-Log scale and at fixed~$h=2^{-6}$ when the time step~$\Delta t$ varies.}
\label{fig:strang}
\end{figure}

We close by noting that the approach to positivity-preservation proposed here does not increase the computational cost associated with solving the resulting systems of equations, as the matrix has the same sparsity structure and can be treated with the same solvers as in standard finite element or finite difference schemes. The only additional requirement is that of a weakly acute triangulation. In the planar case, a weakly acute triangulation can be generated using existing software (see, for example, \cite{MR3765870} and the publications cited therein). In three space dimensions the generation of weakly acute triangulations is more involved (see, for example, \cite{MR4900691} and in particular the discussion on pp.~107 and 108 there; note, though, that those authors use the term \textit{nonobtuse} instead of \textit{weakly acute}). Understanding how
in three space dimensions the method proposed here compares from the viewpoint of \textit{computational cost vs. accuracy} with classical techniques (such as, for example, clipping negative values of the numerical solution to zero) would be a significant computational endeavour, and is beyond
the scope of this paper. Having said this, even if  clipping negative values of the numerical solution to zero happens to be more efficient in practice, the clipping process will certainly impact on the accuracy of the resulting numerical solution and its asymptotic convergence rate to the exact solution as $h \to 0$.  In contrast, the numerical method proposed here exhibits the same \textit{optimal order} of convergence in the limit of $h \to 0$ as its deterministic counterpart for the semilinear parabolic equation $u_t - \Delta u = f(u)$ on $D \times (0,T]$ subject to the initial condition $u(x,0)=u_0(x)$ for $x \in D$ and a homogeneous Dirichlet boundary condition on $\partial D \times (0,T]$.
}

\section*{Acknowledgements}
The authors are grateful to Aur{\'e}lien Alfonsi  and Nicolas Perkowski for informative discussions and to Tony Leli{\`e}vre for his comments on a preliminary version of the manuscript. The numerical experiments in Section~\ref{Section:NumericalExperiments} were conducted by Owen Thomas Hearder (FU Berlin), whom we acknowledge for his contribution.

This research was initiated while the second author was visiting the Berlin Mathematics Research Center MATH+. Partial financial support from the Deutsche Forschungsgemeinschaft (DFG, German Research Foundation) under Germany’s Excellence Strategy---The Berlin Mathematics Research Center MATH+
(EXC-2046/1, project ID: 390685689) is gratefully acknowledged.

The first author gratefully acknowledges funding by the Daimler-Benz Foundation as part of the scholarship program for junior professors and postdoctoral researchers. Further support for the first authors was provided by the Deutsche Forschungsgemeinschaft (DFG), through the research grant CRC 1114 ``Scaling Cascades in Complex Systems'', Project Number 235221301, Project C10. The second author also acknowledges the support of a Senior Zuse fellowship, The Zuse Institute Berlin, and thanks FU Berlin for its hospitality.

\bibliographystyle{abbrv}
\bibliography{ML_FEM_SPDEs_DjLBS_arxiv_2vs}

\end{document}